\documentclass[11pt,oneside]{amsart}
\usepackage[a4paper]{geometry}
\usepackage{amsbsy}
\usepackage{amstext}
\usepackage{amssymb}
 \usepackage{amsaddr}
\usepackage[unicode=true,
 bookmarks=false,
 breaklinks=false,pdfborder={0 0 1},backref=section,colorlinks=false]
 {hyperref}
\usepackage{breakurl}

\makeatletter

\providecommand{\tabularnewline}{\\}

\numberwithin{equation}{section}
\numberwithin{figure}{section}

\usepackage{graphicx}
\usepackage{amsthm}
\usepackage{fancyhdr}
\usepackage{amsfonts}
\usepackage{pgf}
\usepackage{tikz}\usepackage{mathrsfs}
\providecommand{\U}[1]{\protect\rule{.1in}{.1in}}
\usetikzlibrary{arrows}
\newtheorem{theorem}{Theorem}
\theoremstyle{plain}

\newtheorem{proposition}[theorem]{Proposition}
\newtheorem{lemma}[theorem]{Lemma}
\newtheorem{corollary}[theorem]{Corollary}

\theoremstyle{definition}
\newtheorem{definition}[theorem]{Definition}
\theoremstyle{remark}
\newtheorem{remark}[theorem]{Remark}
\numberwithin{equation}{section}

\DeclareMathOperator{\Fi}{Fi}
\DeclareMathOperator{\Cl}{cl}
\DeclareMathOperator{\sat}{sat}
\ifx\pdfoutput\relax\let\pdfoutput=\undefined\fi
\newcount\msipdfoutput
\ifx\pdfoutput\undefined\else
\ifcase\pdfoutput\else
\ifx\paperwidth\undefined\else
\ifdim\paperheight=0pt\relax\else\pdfpageheight\paperheight\fi
\ifdim\paperwidth=0pt\relax\else\pdfpagewidth\paperwidth\fi
\fi\fi\fi

\makeatother

\begin{document}

\title{Monotonic Distributive Semilattices}
\thanks{This paper has received funding from the European Union’s Horizon 2020 research and innovation programme under the Marie Skłodowska-Curie grant agreement No 689176, and the support of the grant PIP 11220150100412CO of CONICET (Argentina).}

\author{Sergio A. Celani \and Ma. Paula Mench\'{o}n}

\email{scelani@exa.unicen.edu.ar}
\email{mpmenchon@exa.unicen.edu.ar}

\address{CONICET and Departamento de Matem\'{a}ticas, Facultad de Ciencias
Exactas, Univ. Nac. del Centro, Pinto 399, 7000 Tandil, Argentina\\
 }

\maketitle
\subjclass{ } 

\begin{abstract}
In the study of algebras related to non-classical
logics, (distributive) semilattices are always present in the background.
For example, the algebraic semantic of the $\{\rightarrow,\wedge,\top\}$-fragment
of intuitionistic logic is the variety of implicative meet-semilattices
\cite{CelaniImplicative} \cite{ChajdaHalasKuhr}. In this paper we
introduce and study the class of distributive meet-semilattices endowed
with a monotonic modal operator $m$. We study the representation
theory of these algebras using the theory of canonical extensions
and we give a topological duality for them. Also, we show
how our new duality extends to some particular subclasses.
\end{abstract}


\section{Introduction}

Boolean algebras with modal operators are the algebraic semantic of
classical modal logics. Using Stone\textquoteright s topological representation
of Boolean algebras, it is known that every Boolean algebra with a
modal operator can be represented as a relational structure \cite{Chagrov-Zakharyaschev}.
This representation plays an important role in the study of many extensions
of normal modal logics \cite{Chagrov-Zakharyaschev} and monotone
modal logics \cite{Chellas} \cite{Hansen}. Recall that monotone
modal logics are a generalization of normal modal logics in which
the axiom $m(\varphi\rightarrow\phi)\rightarrow(m\varphi\rightarrow m\phi)$
has been weakened, leading to a monotonicity condition which can be either
expressed as an axiom ($m(\varphi\wedge\phi)\rightarrow m\varphi$),
or as a rule (from $\varphi\rightarrow\phi$ derive $m\varphi\rightarrow m\phi$).
Thus, it is possible to study monotone modal logics with
a language containing only the connectives $\wedge$ and $m$, or $\rightarrow$
and $m$.

Classical monotone modal logics are interpreted semantically by means
of neighborhood frames \cite{Chellas} \cite{Hansen} \cite{HansenKupkePacuit}.
This class of structures provides a generalization of Kripke semantics. Every neighborhood frame produces a Boolean
algebra endowed with a monotonic operator, called monotonic algebra.
And reciprocally, every monotonic algebra defines a neighborhood frame
(see \cite{Hansen} or \cite{Celaniboole}). Also, it is possible to
consider monotone modal logics defined on non-classical logics. For
example, in \cite{Kojima}, Kojima considered neighborhood semantics
for intuitionistic modal logic, and he defined a neighborhood frame
as a triple $\langle W,\leq,N\rangle$ where $N$ is a neighborhood
function, which is a mapping from $W$ to $\mathcal{P}(\mathcal{P}(W))$
that satisfies the decreasing condition, i.e., $N(x)\supseteq N(y)$
whenever $x\leq y$ (see Definition 3.1 of \cite{Kojima}). Monotonic
logics based on intuitionistic logic are also studied in \cite{Sotirov}.

In the study of algebras related to non-classical logics, semilattices
are always present in the background. For example, the algebraic semantic
of the $\{\rightarrow,\wedge,\top\}$-fragment of intuitionistic logic is the variety of implicative meet-semilattices \cite{CelaniImplicative}
\cite{ChajdaHalasKuhr}, and it is well known that the meet-semilattice
reduct of an implicative meet-semilattice is distributive in the sense
of \cite{Gratzer} or \cite{CelaniTopological}. In \cite{Gratzer}
G. Gr\"{a}tzer gave a topological representation for distributive
semilattices using sober spaces. This representation was extended to
a topological duality in \cite{CelaniTopological} and \cite{CelaniCalomino}.
The principal novelty of \cite{CelaniTopological} was the characterization
of meet-semilattice homomorphisms preserving top by means of certain
binary relations. For implicative semilattices there exists a similar
representation in \cite{CelaniImplicative}. The main objective of
this paper is to study a full Stone style duality for distributive
meet-semilattices endowed with a monotonic operator. So, most of the
results given in this paper are applicable, with minor modifications,
to the study of bounded distributive lattices, implicative semilattices,
 Heyting algebras, and Boolean algebras with monotonic operators.
We note that in the particular case of Boolean algebras our duality
yields the duality given in \cite{Celaniboole} and \cite{Hansen}. 

Canonical extensions were introduced by J\'{o}nsson and Tarski to
study Boolean algebras with operators. The main purpose was to make
it easier to identify what form the dual of an additional operation
on a lattice should take. Since their seminal work, the theory of
canonical extensions has been simplified and generalized \cite{Gehrke -Jonsson2000,PalmigianoDunn},
leading to a theory widely applicable beyond the original Boolean
setting. We will use canonical extension as a tool for the development
of a theory of relational methods, in an algebraic way. 

The paper is organized as follows. In Section 2 we recall the definitions
and some basic properties of distributive semilattices and canonical
extensions. We recall the topological representation and duality developed in \cite{CelaniTopological}
and \cite{CelaniCalomino}. In Section 3, we introduce a special class
of saturated sets of a $DS$-space that is dual to the family of order ideals of a distributive semilattice. In Section 4 we present the class of distributive
semilattices endowed with a monotonic operator, and we extend the
results on representation using canonical extensions. In section 5
we consider some important applications of the duality. We show how
our new duality extends to some particular subclasses.

\section{Preliminaries}

We include some elementary properties of distributive semilattices
that are necessary to read this paper. For more details see \cite{CelaniTopological},
\cite{ChajdaHalasKuhr} and \cite{CelaniCalomino}. 

Let $\langle X,\leq\rangle$ be a poset. For each $Y\subseteq X$,
let $[Y)=\{x\in X:\exists y\in Y(y\leq x)\}$ and $(Y]=\{x\in X:\exists y\in Y(x\leq y)\}$.
If $Y=\{y\}$, then we will write $[y)$ and $(y]$ instead of $[\{y\})$
and $(\{y\}]$, respectively. We call $Y$ an \emph{upset} (resp.
\emph{downset}) if $Y=[Y)$ (resp. $Y=(Y]$). The set of all upsets of $X$ will be denoted by $\mathrm{Up}(X)$. The complement of
a subset $Y\subseteq X$ will be denoted by $Y^{c}$ or $X-Y$. 

\begin{definition} A \emph{meet-semilattice with greatest element}, a
semilattice for short, is an algebra $\mathbf{A}=\langle A,\wedge,1\rangle$
of type $(2,0)$ such that the operation $\wedge$ is idempotent,
commutative, associative, and $a\wedge1=a$ for all $a\in A$. 

\end{definition}

It is clear that for each poset $\langle X,\leq\rangle$ the structure
$\langle\mathrm{Up}(X),\cap,X\rangle$ is a semilattice. 

Let $\mathbf{A}$ be a semilattice. As usual, we can define a partial
order on $\mathbf{A}$, called the natural order, as $a\leq b$ iff
$a\wedge b=a$. It is easy to see that $1$ is the greatest element
of $A$. A subset $F\subseteq A$ is a \emph{filter} of $\mathbf{A}$
if it is an upset, $1\in F$ and if $a,b\in F$, then $a\wedge b\in F$.
We will denote the set of all filters of $\mathbf{A}$ by $\Fi({\mathbf{A}})$.
It is easy to see that $\Fi({\mathbf{A}})$ is closed under arbitrary
intersections. The filter generated by the subset $X\subseteq A$
will be denoted by $F(X)$. If $X=\{a\}$, $F(\{a\})=F(a)=[a)$. We
shall say that a proper filter is \emph{irreducible} or \emph{prime}
if for any pair of filters $F_{1},F_{2}$ such that $F=F_{1}\cap F_{2}$,
it follows that $F=F_{1}$ or $F=F_{2}$.

We will denote the set of all irreducible filters of a semilattice
$\mathbf{A}$ by $X({\mathbf{A}})$. A subset $I\subseteq A$ is called
an \emph{order ideal} if it is a downset and for every $a,b\in I$
we have that there exists $c\in I$ such that $a,b\leq c$. We will
denote the set of all order ideals of $\mathbf{A}$ by $\mathrm{Id}({\mathbf{A}})$. 

\begin{theorem}\cite{CelaniTopological} \label{sep}Let $\mathbf{A}$
be a semilattice. Let $F\in\Fi({\mathbf{A}})$ and let $I\in\mathrm{Id}({\mathbf{A}})$
such that $F\cap I=\emptyset$. Then there exists $P\in X({\mathbf{A}})$
such that $F\subseteq P$ and $P\cap I=\emptyset$. \end{theorem}

A semilattice $\mathbf{A}$ is \emph{distributive} if for all $a,b,c\in A$
such that $a\wedge b\leq c$ there exist $a_{1},b_{1}\in A$ such
that $a\leq a_{1}$, $b\leq b_{1}$ and $c=a_{1}\wedge b_{1}$. We
will denote by $\mathcal{DS}$ the class of distributive semilattices.
We recall (see \cite{CelaniCalomino,CelaniTopological}) that in a distributive semilattice $\mathbf{A}$, if $F$ is a proper filter then the following conditions are equivalent:
\begin{enumerate}
\item $F$ is irreducible,
\item for every $a,b\in A$ such that $a,b\notin F$,
there exist $c\notin F$ and $f\in F$ such that $a\wedge f\leq c$
and $b\wedge f\leq c$,
\item $A-F=F^{c}$ is an order ideal.
\end{enumerate}

Let $\mathbf{A},\mathbf{B}\in \mathcal{DS}$. A mapping $h\colon A\rightarrow B$
is called a \emph{semilattice homomorphism} if 
\begin{enumerate}
\item $h(a)\wedge h(b)=h(a\wedge b)$
for every $a,b\in A$ and
\item $h(1)=1$. 
\end{enumerate} 

Let $\mathbf{A}\in\mathcal{DS}$. Let us consider the poset $\langle X({\mathbf{A}}),\subseteq\rangle$
and the mapping $\beta_{\mathbf{A}}\colon A\rightarrow\mathrm{Up}(X({\mathbf{A}}))$
defined by $\beta_{\mathbf{A}}(a)=\{P\in X({\mathbf{A}}):a\in P\}$.
For convenience, we omit the subscript of $\beta_{\mathbf{A}}$, when
no confusion can arise.

\begin{theorem} \label{rep Hilbert}Let $\mathbf{A}\in\mathcal{DS}$.
Then $\mathbf{A}$ is isomorphic to the subalgebra $\beta[A]=\{\beta(a):a\in A\}$
of $\langle\mathrm{Up}(X({\mathbf{A}})),\cap,X({\mathbf{A}})\rangle$.
\end{theorem}

\subsection{$\mathcal{DS}$-spaces}

In this subsection we recall the duality for distributive semilattices
given in \cite{CelaniCalomino} and \cite{CelaniTopological} based
on a Stone style duality and we give some definitions that we will
need to extend it.

Let $\langle X,\mathcal{T}\rangle$ be a topological space. We will
denote by $\mathcal{KO}(X)$ the set of all compact and open subsets
of $X$ and let $D(X)$ be the set $D(X)=\{U:U^{c}\in\mathcal{KO}(X)\}$.
We will denote by $\mathcal{C}(X)$ the set of all non-empty closed
subsets of $X$. The closure of a subset $Y\subseteq X$ will be denoted
by $\Cl(Y)$. A subset $Y\subseteq X$ is \emph{saturated} if it is
an intersection of open sets. The smallest saturated set containing
$Y$ is the \emph{saturation} of $Y$ and will be denoted by $\sat(Y)$. 

We recall that the \emph{specialization order} of $\langle X,\mathcal{T}\rangle$
is defined by $x\preceq y$ if $x\in\Cl(\{y\})=\Cl(y)$. It is easy
to see that $\preceq$ is a reflexive and transitive relation. If
$X$ is $T_{0}$ then the relation $\preceq$ is a partial order.
The dual order of $\preceq$ will be denoted by $\leq$, i.e., $x\leq y$
if $y\in\Cl(x)$. Moreover, if $X$ is $T_{0}$ then $\Cl(x)=[x)$,
$\sat(Y)=(Y]$, and every open (resp. closed) subset is a downset
(resp. upset) respect to $\leq$.

Recall that a non-empty subset $Y\subseteq X$ of a topological space
$\langle X,\mathcal{T}\rangle$ is \emph{irreducible} if  $Y\subseteq Z\cup W$
for any closed subsets $Z$ and $W$, implies $Y\subseteq Z$ or $Y\subseteq W$.
A topological space $\langle X,\mathcal{T}\rangle$ is \emph{sober}
if for every irreducible closed set $Y\subseteq X$, there exists
a unique $x\in X$ such that $\Cl(x)=Y$. Each sober space is $T_{0}$.
The following definition is equivalent to the definition given by
G. Gr\"{a}tzer in \cite{Gratzer}.

\begin{definition}\cite{CelaniCalomino} A \emph{$DS$-space} is
a topological space $\langle X,\mathcal{T}\rangle$ such that:
\begin{enumerate}
\item The set of all compact and open subsets $\mathcal{KO}(X)$ forms
a basis for the topology $\mathcal{T}$ on $X$.
\item $\langle X,\mathcal{T}\rangle$ is sober. 
\end{enumerate}
\end{definition}

If $\langle X,\mathcal{T}\rangle$ is a $DS$-space, then $\langle D(X),\cap,X\rangle$
is a distributive semilattice (see \cite{Gratzer}).

Let $\langle X,\leq\rangle$ be a poset. Recall that a subset $K\subseteq X$
is called \emph{dually} \emph{directed} if for any $x,y\in K$ there
exists $z\in K$ such that $z\leq x$ and $z\leq y$. A subset $K\subseteq X$
is called \emph{directed} if for any $x,y\in K$ there exists $z\in K$
such that $x\leq z$ and $y\leq z$. 

\begin{theorem} Let $\langle X,\mathcal{T}\rangle$ be a topological
space with basis $\mathcal{K}$ of open and compact subsets for $\mathcal{T}$.
Then, the following conditions are equivalent:
\begin{enumerate}
\item $\langle X,\mathcal{T}\rangle$ is sober 
\item $\langle X,\mathcal{T}\rangle$ is $T_{0}$ and $\bigcap\{U:U\in\mathcal{L}\}\cap Y\neq\emptyset$
for each closed subset $Y$ and for any dually directed subset $\mathcal{L}\subseteq\mathcal{K}$
such that $Y\cap U\neq\emptyset$ for all $U\in\mathcal{L}$.
\end{enumerate}
\end{theorem}

Let $\mathbf{A}\in\mathcal{DS}$. Consider $\mathcal{K}_{\mathbf{A}}=\{\beta(a)^{c}:a\in A\}$
and let $\mathcal{T}_{\mathbf{A}}$ be the topology generated by the
basis $\mathcal{K}_{\mathbf{A}}$. Then, $\langle X({\mathbf{A}}),\mathcal{T}_{\mathbf{A}}\rangle$
is a $DS$-space, called the \emph{dual} \emph{space} of $\mathbf{A}$
(see \cite{CelaniTopological} and \cite{CelaniCalomino}). Recall
that $Q\in\mathrm{\Cl}(P)$ iff $P\subseteq Q$, i.e., the specialization
dual order of $\langle X({\mathbf{A}}),\mathcal{T}_{\mathbf{A}}\rangle$
is the inclusion relation $\subseteq$. Also, recall that the lattices
$\Fi(\mathbf{A})$ and $\mathcal{C}(X(\mathbf{A}))$ are dually isomorphic
under the maps $F\mapsto\hat{F}$, where $\hat{F}=\{P\in X(\mathbf{A}):F\subseteq P\}=\bigcap\{\beta(a):a\in F\}$
for each $F\in\Fi(\mathbf{A})$ and $Y\mapsto F_{Y}$, where $F_{Y}=\{a\in A:Y\subseteq\beta(a)\}$
for each $Y\in\mathcal{C}(X(\mathbf{A}))$.

\subsection{Canonical extension}

Here we will give the basic definitions of the theory of canonical
extensions focused on (distributive) meet semilattices. The following
is an adaptation of the definition given in \cite{PalmigianoDunn}
for posets. This definition agrees with the definition of canonical
extensions for bounded distributive lattices and Boolean algebras
\cite{Gehrke -Jonsson2000,Jonsson y Tarski}.

\begin{definition}Let $\mathbf{A}$ be a a semilattice. A $\mathit{completion}$
of $\mathbf{A}$ is a semilattice embedding $e\colon A\rightarrow X$
where $X$ is a complete lattice. From now on, we will suppress $e$
and call $X$ a completion of $\mathbf{A}$ and assume that $\mathbf{A}$
is a subalgebra of $X$. 

\end{definition}

\begin{definition}Let $\mathbf{A}$ be a semilattice. Given a completion
$X$ of $\mathbf{A}$, an element of $X$ is called $\mathit{closed}$
provided it is the infimum in $X$ of some filter $F$ of $\mathbf{A}$.
We denote the set of all closed elements of $X$ by $K(X)$. Dually,
an element of $X$ is called $\mathit{open}$ provided it is the supremum
in $X$ of some order ideal $I$ of $\mathbf{A}$. We denote the set
of all open elements of $X$ by $O(X)$.  A completion $X$ of $\mathbf{A}$
is said to be $\mathit{dense}$ provided each element of $X$ is both
the supremum of all the closed elements below it and the infimum of
all the open elements above it. A completion $X$ of $\mathbf{A}$
is said to be $\mathit{compact}$ provided that whenever $D$ is a
non-empty dually directed subset of $A$, $U$ is a non-empty directed
subset of $A$, and $\bigwedge_{L}D\leq\bigvee_{L}U$, then there
exist $x\in D$ and $y\in U$ such that $x\leq y$. 

\end{definition}

\begin{definition}Let $\mathbf{A}$ be a semilattice. A $\textit{canonical extension}$
of $\mathbf{A}$ is a dense and compact completion of $\mathbf{A}$.

\end{definition}

\begin{theorem}Let $\mathbf{A}$ be a semilattice, then $\mathbf{A}$
has a canonical extension and it is unique up to an isomorphism that
fixes $\mathbf{A}$. 

\end{theorem}

\begin{lemma}Let us consider a distributive semilattice with greatest
element $\mathbf{A}=\langle A,\wedge,1\rangle$. $\mathrm{\langle Up}(X(\mathbf{A})),\cap,\cup,X(\mathbf{A}),\emptyset\rangle$
is a canonical extension of $\mathbf{A}$, where $A\cong\beta[A]\subseteq\mathrm{Up}(X(\mathbf{A}))$.
We will call it `the' canonical extension.

\end{lemma}

\section{Ideals and saturated subsets}

\label{section: Ideals}

In this section we present a particular family of saturated sets in
a $DS$-space, dual to the family of order ideals of a semilattice.

\begin{definition} Let $\langle X,\mathcal{T}\rangle$ be a $DS$-space.
$Z\subseteq X$ is a \emph{special basic saturated subset} if $Z=\bigcap\{W:W\in\mathcal{L}\}$ for some dually directed family
$\mathcal{L}\subseteq\mathcal{KO}(X)$. \end{definition}

We denote by $\mathcal{S}(X)$ the set of all special basic saturated subsets
of a $DS$-space $\langle X,\mathcal{T}\rangle$. Note that every
special basic saturated subset is a saturated set.

Let $\mathbf{A}\in\mathcal{DS}$. Let $I\in$ $\mathrm{Id}({\mathbf{A}})$.
We consider the following subset of $X({\mathbf{A}}):$
\[
\alpha(I)=\bigcap\{\beta(a)^{c}:a\in I\}=\{P\in X({\mathbf{A}}):I\cap P=\emptyset\}.
\]
It is clear that $\alpha(I)$ is a special basic saturated set of $\langle X({\mathbf{A}}),\mathcal{T}_{\mathbf{A}}\rangle$.
Let $Z\subseteq X({\mathbf{A}})$ be a special basic saturated set of
$X({\mathbf{A}})$. Consider the subset 
\[
I_{\mathbf{A}}(Z)=\{a\in A:\beta(a)\cap Z=\emptyset\}.
\]
It is easy to see that $I_{\mathbf{A}}(Z)$ is a downset of $\mathbf{A}$. 

\begin{remark}The special basic saturated subsets of a $DS$-space
$\langle X,\mathcal{T}\rangle$ are precisely the compact saturated
subsets of the topology.

\end{remark}

Given two posets $\langle X,\leq_{X}\rangle$ and $\langle Y,\leq_{Y}\rangle$,
a \emph{surjective order-isomorphism} from $\langle X,\leq_{X}\rangle$
to $\langle Y,\leq_{Y}\rangle$ is a surjective function $f\colon X\rightarrow Y$
with the property that for every $x$ and $y$ in $X$, $x\leq_{X}y$
if and only if $f(x)\leq_{Y}f(y)$. We say that the posets $\langle X,\leq_{X}\rangle$
and $\langle Y,\leq_{Y}\rangle$ are isomorphic if there exists a surjective
order-isomorphism $f\colon X\rightarrow Y$.

In the following result we prove that order ideals are in bijective
correspondece with the family of basic saturated subsets of $X(\mathbf{A})$.

\begin{theorem} Let $\mathbf{A}\in\mathcal{DS}$. Then the posets
$\langle\mathrm{Id}({\mathbf{A}}),\subseteq\rangle$ and $\langle\mathcal{S}(X(\mathbf{A})),\subseteq\rangle$
are dually isomorphic. \end{theorem}

\begin{proof} Let $Z\subseteq X({\mathbf{A}})$ be a special basic
saturated subset of $X({\mathbf{A}})$. We prove that $I_{\mathbf{A}}(Z)$
is an order ideal of $\mathbf{A}$ and $Z=\alpha(I_{\mathbf{A}}(Z))$.
Moreover, if $I$ is any order ideal of $\mathbf{A}$, then we prove
that $I=I_{\mathbf{A}}(\alpha(I))$.

It is clear that $I_{\mathbf{A}}(Z)$ is a downset of $\mathbf{A}$.
Let $a,b\in I_{\mathbf{A}}(Z)$. So, we have that $Z\cap(\beta_{\mathbf{A}}(a)\cup\beta_{\mathbf{A}}(b))=\emptyset$.
Since $Z=\bigcap\{\beta(a)^{c}:\beta(a)^{c}\in\mathcal{L}\}$ for
some dually directed family $\mathcal{L}\subseteq\mathcal{K}_{\mathbf{A}}$
and $\beta(a)\cup\beta(b)$ is a closed subset, there exists $\beta_{\mathbf{A}}(c)^{c}\in\mathcal{L}$
such that $\beta(c)^{c}\cap(\beta(a)\cup\beta(b))=\emptyset$. Thus,
$\beta(a)\cup\beta(b)\subseteq\beta(c)$ and $Z\cap\beta_{\mathbf{A}}(c)=\emptyset$,
i.e., $a,b\leq c$ and $c\in I_{\mathbf{A}}(Z)$. Therefore, $I_{\mathbf{A}}(Z)$
is an order ideal of $\mathbf{A}$ and we have that $\alpha(I_{\mathbf{A}}(Z))=\bigcap\{\beta(a)^{c}:Z\subseteq\beta(a)^{c}\}\subseteq\bigcap\{\beta(a)^{c}:\beta(a)^{c}\in\mathcal{L}\}=Z$.
The other inclusion is immediate.

Now, let $I$ be an order ideal. Let $b\in I_{\mathbf{A}}(\alpha(I))$.
Then $\beta(b)\cap\alpha(I)=\beta(b)\cap\bigcap\{\beta(a)^{c}:a\in I\}=\emptyset$.
Since $\beta_{\mathbf{A}}(b)$ is a closed subset, and the family
$\{\beta(a)^{c}:a\in I\}$ is dually directed, we get that there exists
$a\in I$ such that $\beta(b)\subseteq\beta(a)$. So, $b\leq a$,
and as $I$ is a downset, we have that $b\in I$. The other inclusion
is immediate.

Thus, we have a surjective function $\alpha\colon\mathrm{Id}({\mathbf{A}})\rightarrow\mathcal{S}(X(\mathbf{A}))$
with inverse function $I_{\mathbf{A}}\colon\mathcal{S}(X(\mathbf{A}))\rightarrow\mathrm{Id}({\mathbf{A}})$.
We prove that $\alpha$ is a dual order-isomorphism. Let $I_{1}$
and $I_{2}$ be two ideals of $\mathbf{A}$. Assume that $I_{1}\subseteq I_{2}$.
Let $P\in\alpha(I_{2})$. Then, $P\cap I_{2}=\emptyset$. It follows
that $P\cap I_{1}=\emptyset$, i.e., $P\in\alpha(I_{1})$.

Assume that $\alpha(I_{1})\subseteq\alpha(I_{2})$. Let $a\in I_{2}$
and suppose that $a\notin I_{1}$. Then $I_{1}\cap[a)=\emptyset$,
so there exists $P\in X({\mathbf{A}})$ such that $[a)\subseteq P$
and $P\cap I_{1}=\emptyset$. It follows that $P\in\alpha(I_{1})$
but $P\notin\alpha(I_{2})$ which is a contradiction. Therefore, $a\in I_{1}$.\end{proof}

\begin{remark} We note that for any $a\in A$, $\alpha((a])=\beta(a)^{c}$.
\end{remark}

For simplicity we will write $\alpha(a)$ instead of $\alpha((a])$. 

\begin{proposition}\label{Sat y cerr}Let $\mathbf{A}\in\mathcal{DS}$, let $Y\in\mathcal{C}(X(\mathbf{A}))$ and $Z\in\mathcal{S}(X(\mathbf{A}))$.
Then,
\[F_{Y}\cap I_{\mathbf{A}}(Z)=\emptyset\text{ iff }Y\cap Z\neq\emptyset.
\] 

\end{proposition}

\begin{proof} Suppose that $F_{Y}\cap I_{\mathbf{A}}(Z)=\emptyset$.
Then, there exists $P\in X(\mathbf{A})$ such that $F_{Y}\subseteq P$
and $P\cap I_{\mathbf{A}}(Z)=\emptyset$, i.e., $P\in Y$ and $P\in Z$.
Thus, $Y\cap Z\neq\emptyset$. The rest of proof is straightforward.\end{proof}

Now, we are able to identify the topological structures that are the
closed and open elements of a canonical extension of a distributive
semilattice. 

\begin{lemma}Let $\mathbf{A}$ be a distributive semilattice. Let
us consider the canonical extension $\langle\mathrm{Up}(X(\mathbf{A})),\cap,\cup,X(\mathbf{A}),\emptyset\rangle$
and the DS-space $\langle X(\mathbf{A}),\mathcal{T}_{\mathbf{A}}\rangle$.
Then, $K(\mathrm{Up}(X(\mathbf{A})))=\mathcal{C}(X(\mathbf{A}))$
and $O(\mathrm{Up}(X(\mathbf{A})))=\{Z^{c}:Z\in\mathcal{S}(X(\mathbf{A}))\}$,
i.e., the closed elements of the canonical extension are exactly the
closed sets of the topology and the open elements of the canonical
extension are the complements of the special saturated sets of the
topology. 

\end{lemma}

\begin{remark}Given a complete lattice $C$, we denote the set of
completely join prime elements by $J^{\infty}(C)$ and the set of completely
meet prime elements by $M^{\infty}(C)$.

Every element of $\mathrm{Up}(X(\mathbf{A}))$ is a join of completely
join prime elements and a meet of completely meet prime elements,
where $J^{\infty}(\mathrm{Up}(X(\mathbf{A})))=\{\hat{P}=[P):P\in X(\mathbf{A})\}$
and $M^{\infty}(\mathrm{Up}(X(\mathbf{A})))=\{\alpha(P^{c})^{c}=(P]^{c}:P\in X(\mathbf{A})\}$.

\end{remark}

\section{Representation and duality of monotonic distributive semilattices}

\begin{definition} Let $\mathbf{A}=\langle A,\wedge,1\rangle$ be
a distributive semilattice. A \textit{monotonic operator }is an operator
$m:A\rightarrow A$ that satisfies the following condition 
\[
\text{If }a\leq b\text{, then }ma\leq mb\text{ for all }a,b\in A.
\]

\end{definition}

The following result is immediate.

\begin{proposition} Let $\mathbf{A}\in\mathcal{DS}$ and let $m:A\rightarrow A$
be a unary function. Then the following conditions are equivalent:
\begin{enumerate}
\item For all $a,b\in A$, if $a\leq b$ then $ma\leq mb$,
\item $m(a\wedge b)\leq ma\wedge mb$ for all $a,b\in A$. 
\end{enumerate}
\end{proposition}

\begin{definition} Let $\mathbf{A}\in\mathcal{DS}$. The pair $\langle\mathbf{A},m\rangle$
such that $m$ is a monotonic operator defined on $\mathbf{A}$ is
called a \emph{monotonic distributive semilattice}. \end{definition}

The class of all monotonic distributive semilattices will be denoted by
$\mathcal{MDS}$.

Let $\langle\mathbf{A},m\rangle,\langle\mathbf{B},m\rangle\in\mathrm{\mathcal{MDS}}$.
We say that a homomorphism $h$ $\colon A\rightarrow B$ is a \emph{homomorphism
of monotonic distributive semilattices} if $h$ commutes with $m$,
i.e., if $h(ma)=mh(a)$ for all $a\in A$. We denote by $\mathcal{MDSH}$
the category of monotonic distributive semilattices and monotonic
distributive semilattice homomorphisms.

We will give two examples of monotonic distributive semilattices constructed
from certain relational systems. We shall use these examples in the
theory of representation and topological duality for monotonic distributive
semilattices.

Let $X$ be a set. A \emph{multirelation} on $X$ is a subset of the
Cartesian product $X\times\mathcal{P}(X)$, that is, a set of ordered
pairs $(x,Y)$ where $x\in X$ and $Y\subseteq X$ \cite{Duntsch-Orlowska-RewitzkyMultirelations2010}
\cite{Rewitzky2003}. We recall that in classical monotone modal logic
a neighborhood frame is a pair $\langle X,R\rangle$ where $X$ is
a set and $R\subseteq X\times\mathcal{P}(X)$, i.e., $R$ is a multirelation
(see \cite{Chellas} \cite{Hansen}). Now we give a generalization
of this notion.

\begin{definition} \label{SCneighborhood}An $S$\emph{-neighborhood}
frame is a triple $\langle X,\leq,R\rangle$ where $\langle X,\leq\rangle$
is a poset and $R$ is a subset of $X\times\mathcal{P}(X)$ such that
if $x\leq y$, then $R(y)\subseteq R(x)$ for all $x,y\in X$. For
each $U\in\mathrm{Up}(X)$ we define the set 
\begin{equation}
m_{R}(U)=\{x\in X:\forall Z\in R(x)~(Z\cap U\neq\emptyset)\}.\label{eq:op1}
\end{equation}
A $C$\emph{-neighborhood} frame is a triple $\langle X,\leq,G\rangle$
where $\langle X,\leq\rangle$ is a poset and $G$ is a subset of
$X\times\mathcal{P}(X)$ such that if $x\leq y$, then $G(x)\subseteq G(y)$
for all $x,y\in X$. For each $U\in\mathrm{Up}(X)$ we define the set
\begin{equation}
\mathbf{m}_{G}(U)=\{x\in X:\exists Y\in G(x)~(Y\subseteq U)\}.\label{eq:op2}
\end{equation}

\end{definition} 

\begin{lemma} Let $\langle X,\leq,R\rangle$ be an $S$-neighborhood frame and $\langle X,\leq,G\rangle$ be a $C$-neighborhood frame.
Then $\langle\mathrm{Up}(X),\cap,m_{R},X\rangle$ and $\langle\mathrm{Up}(X),\cap,\mathbf{m}_{G},X\rangle$
are monotonic distributive semilattices. \end{lemma}

We will represent the monotonic operator $m$ on a distributive semilattice
$\mathbf{A}$ as a multirelation on the dual space of $\mathbf{A}$,
where the canonical extension offers an advantageous point of view.
We consider two different ways of extending maps that agree with the
ones given in \cite{PalmigianoDunn} for posets, bounded distributive
lattices and Boolean algebras.

\begin{definition}Let $\mathbf{A}$ be a distributive semilattice.
Given a monotonic operation $m\colon A\rightarrow A$, we define the
maps 
\[
m^{\sigma},m^{\pi}:\mathrm{Up}(X(\mathbf{A}))\rightarrow\mathrm{Up}(X(\mathbf{A}))
\]
 by
\[
m^{\sigma}(X)=\bigcup\{\bigcap\{\beta(ma):Y\subseteq\beta(a)\}:X\supseteq Y\in\mathcal{C}(X(\mathbf{A}))\}
\]
and
\[
m^{\pi}(X)=\bigcap\{\bigcup\{\beta(ma):Z\subseteq\beta(a)^{c}\}:X^{c}\supseteq Z\in\mathcal{S}(X(\mathbf{A}))\}.
\]

\end{definition}

The two extensions of a map $m$ shown above are not always equal. Whether
we want to extend a particular additional operation using the $\sigma$-
or the $\pi$-extension depends on the properties of the particular
operation to be extended. The following lemma is a consequence of
Lemma 3.4 of \cite{PalmigianoDunn}.

\begin{lemma}\label{lemaim}Let $\langle\mathbf{A},m\rangle\in\mathcal{MDS}$.
The maps $m^{\sigma},m^{\pi}$ are monotonic extensions of $m$, i.e.,
$\langle\mathrm{Up}(X(\mathbf{A})),m^{\sigma}\rangle,\langle\mathrm{Up}(X(\mathbf{A})),m^{\pi}\rangle\in\mathcal{MDS}$
and $m^{\sigma}(\beta(a))=m^{\pi}(\beta(a))=\beta(ma)$ for all $a\in A$.
In addition, $m^{\sigma}\leq m^{\pi}$ with equality holding on $K(\mathrm{Up}(X(\mathbf{A})))\cup O(\mathrm{Up}(X(\mathbf{A})))$.
For $X\in\mathrm{Up}(X(\mathbf{A}))$, $Y\in\mathcal{C}(X(\mathbf{A}))$
and $Z\in\mathcal{S}(X(\mathbf{A}))$
\[\begin{array}{lll}
m^{\sigma}(X)&=&\bigcup\{m^{\sigma}(Y):X\supseteq Y\in\mathcal{C}(X(\mathbf{A}))\},\\
m^{\sigma}(Y)&=&\bigcap\{\beta(ma):Y\subseteq\beta(a)\},\\
m^{\pi}(X)&=&\bigcap\{m^{\pi}(Z^{c}):X^{c}\supseteq Z\in\mathcal{S}(X(\mathbf{A}))\},\\
m^{\pi}(Z^{c})&=&\bigcup\{\beta(ma):Z\subseteq\beta(a)^{c}\}.
\end{array}\]
So, $m^{\sigma}$ and $m^{\pi}$ send closed sets to closed sets and complements
of special saturated sets to complements of special saturated sets. 

\end{lemma}

Now we show how, using the $\sigma$-extension and the $\pi$-extension,
it is possible to define two multirelations on the dual space of $\mathbf{A}$. 

Let $\langle\mathbf{A},m\rangle\in\mathcal{MDS}$. Note that by
definition of $m^{\pi}$, for every $Z\in\mathcal{S}(X(\mathbf{A}))$ we have:

\begin{center}

\begin{tabular}{lll}
$P\in m^{\pi}(Z^{c})$ & $\Leftrightarrow$ &  $\exists a\in A$ such that $P\in\beta(ma)$ and $Z\subseteq\beta(a)^{c}$ \tabularnewline
 & $\Leftrightarrow$ & $\exists a\in A$ such that $ma\in P$ and $a\in I_{\mathbf{A}}(Z)$ \tabularnewline
 & $\Leftrightarrow$ & $m^{-1}(P)\cap I_{\mathbf{A}}(Z)\neq\emptyset$.\tabularnewline
\end{tabular}

\end{center}

So, for every $X\in\mathrm{Up}(X(\mathbf{A}))$ we get:

\begin{center}

\begin{tabular}{lll}
$P\in m^{\pi}(X)$ & $\Leftrightarrow$ & $\forall Z\in\mathcal{S}(X(\mathbf{A}))$ such that $Z\subseteq X^{c}$,
we have $P\in m^{\pi}(Z^{c})$\tabularnewline
 & $\Leftrightarrow$ & $\forall Z\in\mathcal{S}(X(\mathbf{A}))$ such that $Z\subseteq X^{c}$, we have $m^{-1}(P)\cap I_{\mathbf{A}}(Z)\neq\emptyset$. \tabularnewline
\end{tabular}

\end{center}

We define the relation 
\[
R_{m}\subseteq X(\mathbf{A})\times\mathcal{S}(X(\mathbf{A}))
\]
 by 
\begin{equation}
(P,Z)\in R_{m}\text{~iff~}m^{-1}(P)\cap I_{\mathbf{A}}(Z)=\emptyset.\label{eq:Relation sat}
\end{equation}
Consequently, the operation $m^{\pi}$ on $\mathrm{Up}(X(\mathbf{A}))$
can be defined in terms of the relation $R_{m}$ as:

\[
P\in m^{\pi}(X)\text{\text{~}iff }\text{~}\forall Z\in R_{m}(P)[Z\cap X\neq\emptyset].
\]

On the other hand, we can define another multirelation using the operation
$m^{\sigma}$. Note that for each $Y\in\mathcal{C}(X(\mathbf{A}))$
we have:

\begin{center}

\begin{tabular}{lcl}
$P\in m^{\sigma}(Y)$  & $\Leftrightarrow$ & $\forall a\in A$ such that $Y\subseteq\beta(a)$ we get $P\in\beta(ma)$\tabularnewline
 & $\Leftrightarrow$ & $\forall a\in A$ such that $a\in F_{Y}$ we get $ma\in P$\tabularnewline
 & $\Leftrightarrow$ & $F_{Y}\subseteq m^{-1}(P)$.\tabularnewline
\end{tabular}

\end{center}

So, for each $X\in\mathrm{Up}(X(\mathbf{A}))$ we obtain:

\begin{center}

\begin{tabular}{lll}
$P\in m^{\sigma}(X)$  & $\Leftrightarrow$ & $\exists Y\in\mathcal{C}(X(\mathbf{A}))$ such that $Y\subseteq X$
and $P\in m^{\sigma}(Y)$\tabularnewline
 & $\Leftrightarrow$ & $\exists Y\in\mathcal{C}(X(\mathbf{A}))$ such that $Y\subseteq X$
and $P\in m^{\sigma}(Y)$\tabularnewline
 & $\Leftrightarrow$ & $\exists Y\in\mathcal{C}(X(\mathbf{A}))$ such that $Y\subseteq X$
and $F_{Y}\subseteq m^{-1}(P)$.\tabularnewline
\end{tabular} 

\end{center}

Thus, we can define another relation $G_{m}\subseteq X(\mathbf{A})\times\mathcal{C}(X(\mathbf{A}))$
as 
\[
(P,Y)\in G_{m}\text{~iff~}F_{Y}\subseteq m^{-1}(P).
\]
Consequently, the operation $m^{\sigma}$ on $\mathrm{Up}(X(\mathbf{A}))$
can be defined in terms of the relation $G_{m}$ as:

\begin{equation}
P\in m^{\sigma}(X)\text{~iff~}\exists Y\in G_{m}(P)[Y\subseteq X].\label{eq:Relation clo}
\end{equation}

\begin{remark}Let $\langle\mathbf{A},m\rangle\in MDS$. Note that
$\langle X(\mathbf{A}),\subseteq,R_{m}\rangle$ and $\langle X(\mathbf{A}),\subseteq,G_{m}\rangle$
are an S-monotonic and a C-monotonic frame, respectively, where $m_{R_{m}}=m^{\pi}$
and $\mathbf{m}_{G_{m}}=m^{\sigma}$.

\end{remark}

Now, we are able to define the dual spaces of monotonic distributive
semilattices. Depending on the way we define the relation on the dual
space, there are two possible constructions of relational systems.
However, we will show that both systems are interdefinible. For each
monotonic operator, we can choose to work with either of them based
on its behavior. In the next section we will see how some additional
conditions affect the relations associated. 

Let $\langle X,\mathcal{K}\rangle$ be a $DS$-space. For each $U\in D(X)$
we define the subsets $L_{U}$ and $D_{U}$ of $\mathcal{P}(\mathcal{S}(X))$
and $\mathcal{P}(\mathcal{C}(X))$ as follows:
\[
L_{U}=\{Z\in\mathcal{S}(X):Z\cap U\neq\emptyset\}
\]
and 
\[
D_{U}=\{Y\in\mathcal{C}(X):Y\subseteq U\}.
\]

\begin{definition} \label{t esp mon} An \textit{$\mathcal{S}$-monotonic
}$DS$\textit{-space} is a structure $\langle X,\mathcal{T},R\rangle$,
where $\langle X,\mathcal{T}\rangle$ is a $DS$-space, and $R\subseteq X\times\mathcal{S}(X)$
is a multirelation such that 
\begin{enumerate}
\item $m_{R}(U)=\{x\in X:\forall Z\in R(x)[Z\cap U\neq\emptyset]\}\in D(X)$,
for all $U\in D(X)$ and
\item $R(x)=\bigcap\{L_{U}:U\in D(X)\text{ and }x\in m_{R}(U)\}$, for all
$x\in X$.
\end{enumerate}
We can also give an analogous definition of $\mathcal{C}$-\textit{monotonic
}$DS$\textit{-space} as a structure $\left\langle X,\mathcal{T},G\right\rangle $,
where $\left\langle X,\mathcal{T}\right\rangle $ is a $DS$-space
and $G\subseteq X\times\mathcal{C}(X)$ is a multirelation such that 
\begin{enumerate}
 \setcounter{enumi}{2}
\item $\mathbf{m}_{G}\left(U\right)=\{x\in X:\exists Y\in G\left(x\right)\left[Y\subseteq U\right]\}\in D(X)$
for all $U\in D(X)$, and
\item $G\left(x\right)={\textstyle \bigcap}\{(D_{U})^{c}:U\in D(X)\text{ and }x\in \mathbf{m}_{G}\left(U\right)^{c}\}$
for all $x\in X$.
\end{enumerate}
\end{definition}

\begin{lemma} Let $\langle X,\mathcal{T},R\rangle$ and $\left\langle X,\mathcal{T},G\right\rangle $
be an \textit{$\mathcal{S}$}-monotonic $DS$-space and a\textit{
$\mathcal{C}$}-monotonic $DS$-space respectively. Then, 
\begin{enumerate}
\item $R(y)\subseteq R(x)$ for all $x,y\in X$ such that $x\leq y$ and
\item $G(x)\subseteq G(y)$ for all $x,y\in X$ such that $x\leq y$.
\end{enumerate}
\end{lemma}

\begin{proof}1. Suppose that $x\leq y$ and let $Z\in R(y)$. Let $U\in D(X)$
such that $x\in m_{R}(U)$. By (1) of Definition \ref{t esp mon},
$m_{R}(U)$ is an upset, then $y\in m_{R}(U)$. By (2) of Definition \ref{t esp mon} we have that
$Z\cap U\neq\emptyset$. Then, $Z\in\bigcap\{L_{U}:U\in D(X)\text{ and }x\in m_{R}(U)\}=R(x)$.

2. Suppose that $x\leq y$. Let $Y\in G(x)$. Let $U\in D(X)$ such
that $y\in \mathbf{m}_{G}(U)^{c}$. By (3) of Definition \ref{t esp mon},
$\mathbf{m}_{G}(U)^{c}$ is a downset, then $x\in \mathbf{m}_{G}(U)^{c}$. By (4) of Definition \ref{t esp mon} we
have that $Y\cap U^{c}\neq\emptyset$. Then, $Y\in\bigcap\{(D_{U})^{c}:U\in D(X)\text{ and }y\in \mathbf{m}_{G}(U)^{c}\}=G(x)$.\end{proof}

As a corollary we have that $\langle X,\leq,R\rangle$ is an S-neighborhood frame and $\langle X,\leq,G\rangle$
is a C-neighborhood frame. From Definition \ref{SCneighborhood},
we get that the algebras $\langle\mathrm{Up}(X),m_{R}\rangle$ and
$\langle\mathrm{Up}(X),\mathbf{m}_{G}\rangle$, considering the operators defined
by \ref{eq:op1} and \ref{eq:op2}, are monotonic distributive semilattices.
In consequence, by (1) and (3) of Definition \ref{t esp mon}, $\langle D(X),m_{R}\rangle$
and $\langle D(X),\mathbf{m}_{G}\rangle$ are monotonic distributive semilattices
considering the operators restricted to $D(X)$. 

Now, we will see how we get a kind of space from the other.

\begin{definition} Let $\left\langle X,\mathcal{T}\right\rangle $
be a $DS$-space. Let $\phi_{X}\colon\mathcal{P}(\mathcal{S}(X))\rightarrow\mathcal{P}(\mathcal{C}(X))$
be the function defined by 
\[
\phi_{X}(S)=\{Y\in\mathcal{C}(X):\forall Z\in S\ [Y\cap Z\neq\emptyset]\}
\]
 and let $\psi_{X}\colon\mathcal{P}(\mathcal{C}(X))\rightarrow\mathcal{P}(\mathcal{S}(X))$
be the function defined by 
\[
\psi_{X}(C)=\{Z\in\mathcal{S}(X):\forall Y\in C\ [Y\cap Z\neq\emptyset]\}.
\]

\end{definition}

It is easy to see that 
\[
C\subseteq\phi_{X}(S)\text{ iff }S\subseteq\psi_{X}(C),
\]
 for all $S\subseteq\mathcal{S}(X)$ and $C\subseteq\mathcal{C}(X)$.
It follows that the pair $(\phi_{X},\psi_{X})$ is a Galois connection. 

\begin{proposition} 
\begin{enumerate}
\item Given an $\mathcal{S}$-monotonic $DS$-space $\langle X,\mathcal{T},R\rangle$,
the relation $G_{R}\subseteq X\times\mathcal{C}(X)$ defined as 
\[
\left(x,Y\right)\in G_{R}\text{ iff }Y\in\phi_{X}(R(x))
\]
is such that $\langle X,\mathcal{T},G_{R}\rangle$ is a $\mathcal{C}$-monotonic
$DS$-space and $m_{R}(U)=\mathbf{m}_{G_{R}}(U)$ for all $U\in D(X)$.
\item Given a $\mathcal{C}$-monotonic $DS$-space $\langle X,\mathcal{T},G\rangle$,
the relation $R_{G}\subseteq X\times\mathcal{S}(X)$ defined as 
\[
\left(x,Z\right)\in R_{G}\text{ iff }Z\in\psi_{X}(G(x))
\]
is such that $\langle X,\mathcal{T},R_{G}\rangle$ is a $\mathcal{S}$-monotonic
$DS$-space and $\mathbf{m}_{G}(U)=m_{R_{G}}(U)$ for all $U\in D(X)$. 
\end{enumerate}
\end{proposition}

\begin{proof} 1. Let $\langle X,\mathcal{T},R\rangle$ be an $\mathcal{S}$-monotonic
$DS$-space. We will see that $m_{R}(U)=\mathbf{m}_{G_{R}}(U)$ for all $U\in D(X)$.
Let $U\in D(X)$ and $x\in m_{R}(U)$. Then, for all $Z\in R(x)$
we have that $Z\cap U\neq\emptyset$ and since $U\in\mathcal{C}(X)$,
$U\in G_{R}(x)$. From $U\subseteq U$, we get that $x\in \mathbf{m}_{G_{R}}(U)$.
Now, suppose that $x\in \mathbf{m}_{G_{R}}(U)$. So, there exists $Y\in G_{R}(x)$
such that $Y\subseteq U$ and since for all $Z\in R(x)$ we have that
$Y\cap Z\neq\emptyset$, then $Z\cap U\neq\emptyset$ for all $Z\in R(x)$
and thus $x\in m_{R}(U)$.

We have proved that $m_{R}(U)=\mathbf{m}_{G_{R}}(U)$ and since $\langle X,T,R\rangle$
is an $\mathcal{S}$-monotonic $DS$-space, $\mathbf{m}_{G_{R}}(U)\in D(X)$.

Now, we will see that $G_{R}\left(x\right)=\bigcap\{(D_{U})^{c}:U\in D(X)\text{ and }x\in \mathbf{m}_{G_{R}}\left(U\right)^{c}\}$
for all $x\in X$. Let $x\in X$. It is easy to prove the inclusion
$G_{R}\left(x\right)\subseteq\bigcap\{(D_{U})^{c}:U\in D(X)\text{ and }x\in \mathbf{m}_{G_{R}}\left(U\right)^{c}\}$.
To prove the other inclusion, let $Y\in\bigcap\{(D_{U})^{c}:U\in D(X)\text{ and }x\in \mathbf{m}_{G_{R}}\left(U\right)^{c}\}$
and suppose that $Y\notin G_{R}(x)$. So, there exists $Z\in R(x)$
such that $Z\cap Y=\emptyset$. Since $Z\in\mathcal{S}(X)$ and $Y\in\mathcal{C}(X)$,
there exists $U\in D(X)$ such that $Z\subseteq U^{c}$ and $Y\cap U^{c}=\emptyset$.
Then $Z\cap U=\emptyset$, $x\notin m_{R}(U)=\mathbf{m}_{G_{R}}(U)$ and $Y\subseteq U$,
i.e., $Y\in D_{U}$, which is a contradiction.

2. Let $\langle X,\mathcal{T},G\rangle$ be a $\mathcal{C}$-monotonic
$DS$-space. We will see that $\mathbf{m}_{G}(U)=m_{R_{G}}(U)$ for all $U\in D(X)$.
Let $U\in D(X)$ and $x\in m_{R_{G}}(U)$. Suppose that $x\notin \mathbf{m}_{G}(U)$.
Then, for all $Y\in G(x)$, $Y\cap U^{c}\neq\emptyset$. So, $U^{c}\in R_{G}(x)$
which contradicts the fact that $x\in m_{R_{G}}(U)$. Now, suppose
that $x\in \mathbf{m}_{G}(U)$. Then there exists $Y\in G(x)$ such that $Y\subseteq U$.
Let $Z\in R_{G}(x)$. So, we have that $Y\cap Z\neq\emptyset$, then
$Z\cap U\neq\emptyset$. Thus $x\in m_{R_{G}}(U)$.

We have proved that $\mathbf{m}_{G}(U)=m_{R_{G}}(U)$ and since $\langle X,\mathcal{T},G\rangle$
is a $\mathcal{C}$-monotonic $DS$-space, $m_{R_{G}}(U)\in D(X)$.

Now, we will see that $R_{G}\left(x\right)=\bigcap\{L_{U}:U\in D(X)\text{ such that }x\in m_{R_{G}}(U)\}$
for all $x\in X$. Let $x\in X$. The proof of the inclusion $R_{G}\subseteq\bigcap\{L_{U}:U\in D(X)\text{ such that }x\in m_{R_{G}}(U)\}$
is easy. Let $Z\in\bigcap\{L_{U}:U\in D(X)\text{ such that }x\in m_{R_{G}}(U)\}$
and suppose that $Z\notin R_{G}(x)$. So, there exists $Y\in G(x)$
such that $Z\cap Y=\emptyset$. Since $Z\in\mathcal{S}(X)$ and $Y\in\mathcal{C}(X)$,
there exists $U\in D(X)$ such that $Z\subseteq U^{c}$ and $Y\cap U^{c}=\emptyset$.
Then, $Y\subseteq U$, $x\in \mathbf{m}_{G}(U)=m_{R_{G}}(U)$ and $Z\cap U=\emptyset$,
i.e., $Z\notin L_{U}$, which is a contradiction. \end{proof}

\begin{proposition} \label{prop dual Hilbert space} Let $\langle\mathbf{A},m\rangle\in\mathcal{MDS}$.
Then $\langle X({\mathbf{A}}),\mathcal{T}_{\mathbf{A}},R_{m}\rangle$
is an $\mathcal{S}$-monotonic $DS$-space and $\langle X({\mathbf{A}}),\mathcal{T}_{\mathbf{A}},G_{m}\rangle$
is a $\mathcal{C}$-monotonic $DS$-space. \end{proposition}

\begin{proof} Let $U\in D(X({\mathbf{A}}))$. By definition, $U=\beta(a)$
for some $a\in A$. By Lemma \ref{lemaim} we have that $m_{R_{m}}(\beta(a))=\mathbf{m}_{G_{m}}(\beta(a))=\beta(ma)\in D(X({\mathbf{A}}))$,
i.e., $m_{R_{m}}(U),\mathbf{m}_{G_{m}}(U)\in D(X({\mathbf{A}}))$ for all $U\in D(X({\mathbf{A}}))$. 

Now we will show that for all $P\in X({\mathbf{A}})$ 
\[
R_{m}(P)=\bigcap\{L_{\beta(a)}:ma\in P\}.
\]
Let $P\in X({\mathbf{A}})$. It is clear that $R_{m}(P)\subseteq\bigcap\{L_{\beta(a)}:ma\in P\}$.
On the other hand, let $Z\in\bigcap\{L_{\beta(a)}:ma\in P\}$, we
will prove that $Z\in R_{m}(P)$. Suppose, contrary to our claim,
that $Z\notin R_{m}(P)$. Then, there exists $a\in m^{-1}(P)$ such
that $Z\cap\beta(a)=\emptyset$. By assumption, $Z\in L_{\beta(a)}$,
i.e., $Z\cap\beta(a)\neq\emptyset$, which is a contradiction. Therefore,
$Z\in R_{m}(P)$.

The indentity $G_{m}\left(P\right)={\textstyle \bigcap}\{(D_{\beta(a)})^{c}:ma\notin P\}$
is proved similarly. \end{proof}

\begin{lemma}\label{rels}

Let $\langle\mathbf{A},m\rangle\in\mathcal{MDS}$. Then $R_{m}(P)=\psi_{X(\mathbf{A})}(G_{m}(P))$
and $G_{m}(P)=\phi_{X(\mathbf{A})}(R_{m}(P))$. Therefore the sets
$R_{m}(P)$ and $G_{m}(P)$ are closed sets of the Galois connection $(\phi_{X(\mathbf{A})},\psi_{X(\mathbf{A})})$.

\end{lemma}

\begin{proof} First, we will prove that $R_{m}(P)=\psi_{X(\mathbf{A})}(G_{m}(P))$.
Let $Z\in R_{m}(P)$. Then, $m^{-1}(P)\cap I_{\mathbf{A}}(Z)=\emptyset$.
Let $Y\in G_{m}(P)$. By definition, $F_{Y}\subseteq m^{-1}(P)$ and
we get that $F_{Y}\cap I_{\mathbf{A}}(Z)=\emptyset$. From Proposition
\ref{Sat y cerr}, $Y\cap Z\neq\emptyset$. Now, let $Z\in S(X(\mathbf{A}))$
and suppose that for all $Y\in G_{m}(P)$, $Z\cap Y\neq\emptyset$.
Let $a\in m^{-1}(P)\cap I_{\mathbf{A}}(Z)$. So, $[a)\subseteq m^{-1}(P)$
and by hypothesis $\widehat{[a)}\cap Z=\beta(a)\cap Z\neq\emptyset$.
Then $a\notin I_{\mathbf{A}}(Z)$ which is a contradiction. 

The other equality is proved analogously.\end{proof} 

From now on, we consider a monotonic $DS$-space as an $\mathcal{S}$-monotonic
$DS$-space. It is clear how we can construct one kind of space from
the other, and there is no particular reason we have chosen $\mathcal{S}$-monotonic
$DS$-spaces as our default other than to keep things simple and avoid
repetition obtaining similar theorems and propositions.

\begin{definition}Given $\langle\mathbf{A},m\rangle\in\mathcal{MDS}$,
the structure $\langle X({\mathbf{A}}),\mathcal{T}_{\mathbf{A}},R_{m}\rangle$
\emph{is the monotonic $DS$-space associated to} $\langle\mathbf{A},m\rangle$.

\end{definition}

\begin{definition}

The algebra $\langle D(X),m_{R}\rangle$ is the \emph{monotonic distributive
semilattice associated to}\textit{ }the monotonic $DS$-space $\langle X,\mathcal{T},R\rangle$.

\end{definition}

Now, we are able to enunciate the representation theorem.

\begin{theorem}[of Representation]\label{reo mon}Let $\langle\mathbf{A},m\rangle\in\mathcal{MDS}$.
Then, the structure \linebreak{}
$\langle\mathrm{Up}(X({\mathbf{A}})),\cap,m_{R_{m}},X({\mathbf{A}})\rangle$
is a monotonic distributive semilattice and the map $\beta\colon A\rightarrow\mathrm{Up}(X({\mathbf{A}}))$
defined by 
\[
\beta(a)=\{P\in X({\mathbf{A}}):a\in P\}
\]
is an injective homomorphism of monotonic distributive semilattices.
\end{theorem}

\begin{proof} It follows from Theorem \ref{rep Hilbert} and the
fact that for all $a\in A$, $m_{R_{m}}(\beta(a))=\beta(ma)$. \end{proof}

\begin{corollary} \label{rep cor}Let $\langle\mathbf{A},m\rangle\in\mathcal{MDS}$.
Then, the map $\beta\colon A\rightarrow D(X(\mathbf{A}))$ defined
by 
\[
\beta(a)=\{P\in X({\mathbf{A}}):a\in P\}
\]
is an isomorphism of monotonic distributive  semilattices. \end{corollary}

We note that if $\langle X,\mathcal{T},R\rangle$ is a monotonic $DS$-space,
then we have that \linebreak{}
$\langle X(D(X)),\mathcal{T}_{D(X)},R_{m_{R}}\rangle$ is the monotonic
space associated to $\langle D(X),m_{R}\rangle$. In \cite{CelaniTopological}
Celani has proved that the map 
\[
H_{X}:X\rightarrow X(D(X))
\]
 defined by 
\[
H_{X}(x)=\{U\in D(X):x\in U\},
\]
is an homeomorphism between $DS$-spaces and an order isomorphism
with respect to $\leq$. Now we introduce the following definition.

\begin{definition} Let $\langle X_{1},\mathcal{T}_{1},R_{1}\rangle$
and $\langle X_{2},\mathcal{T}_{2},R_{2}\rangle$ be two monotonic
$DS$-spaces. A map $f\colon X_{1}\rightarrow X_{2}$ is an \emph{isomorphism
of }$DS$\emph{-spaces} if it satisfies,
\begin{enumerate}
\item $f$ is a homeomorphism,
\item $(x,Z)\in R_{1}$ if and only if $(f(x),f[Z])\in R_{2}$, for all
$x\in X_{1}$ and for each $Z\in\mathcal{S}(X_{1})$,
\end{enumerate}
where $f[Z]=\{f(z):z\in Z\}$.

\end{definition}

\begin{proposition}Let $\langle X_{1},\mathcal{T}_{1}\rangle$ and
$\langle X_{2},\mathcal{T}_{2}\rangle$ be two $DS$-spaces and let
$f\colon X_{1}\rightarrow X_{2}$ be a homeomorphism. Then, $f[Z]\in\mathcal{S}(X_{2})$
for all $Z\in\mathcal{S}(X_{1})$ and for all $S\in\mathcal{S}(X_{2})$
there exists $Z\in\mathcal{S}(X_{1})$ such that $S=f[Z]$. \end{proposition}

\begin{remark} \label{prop Hx} Let $\langle X,\mathcal{T}\rangle$
be a $DS$-space. Then, $H_{X}[Z]\in\mathcal{S}(X(D(X))$ for all
$Z\in S(X)$ and for all $S\in\mathcal{S}(X(D(X))$ there exists $Z\in\mathcal{S}(X)$
such that $S=H_{X}[Z]$. Also, we have that $H_{X}[U]=\{H_{X}(u):u\in U\}=\beta_{D(X)}(U)$
for all $U\in D(X)$. Then, 
\[
Z\cap U=\emptyset\Leftrightarrow H_{X}[Z]\cap\beta_{D(X)}(U)=\emptyset
\]
 for all $Z\in\mathcal{S}(X)$ and $U\in D(X)$.

\end{remark}

\begin{theorem} \label{HX}Let $\langle X,\mathcal{T},R\rangle$
be a monotonic $DS$-space. Then, the map $H_{X}\colon X\rightarrow X(D(X))$
defined by 
\[
H_{X}(x)=\{U\in D(X):x\in U\}
\]
is an isomorphism of monotonic $DS$-spaces. \end{theorem}

\begin{proof} By \cite{CelaniTopological} and \cite{CelaniCalomino},
it is only left to prove that $(x,Z)\in R$ iff $(H_{X}(x),H_{X}[Z])\in R_{m_{R}}$.

$\Rightarrow)$ Let $Z\in\mathcal{S}(X)$ such that $(x,Z)\in R$.
We will see that $H_{X}[Z]\cap\beta_{D(X)}(U)\neq\emptyset$ for all
$U\in m_{R}^{-1}(H_{X}(x))$. Let $U\in D(X)$ such that $U\in m_{R}^{-1}(H_{X}(x))$,
i.e., $m_{R}(U)\in H_{X}(x)$. Then, $x\in m_{R}(U)$. Since $Z\in R(x)$,
we get that $Z\cap U\neq\emptyset$ and by Remark \ref{prop Hx},
$H_{X}[Z]\cap\beta_{D(X)}(U)\neq\emptyset$. Therefore, $U\notin I_{D(X)}(H_{X}[Z])$.

$\Leftarrow)$ Suppose that $(H_{X}(x),H_{X}[Z])\in R_{m_{R}}$. Then,
$H_{X}[Z]\cap\beta_{D(X)}(U)\neq\emptyset$ for all $U\in m_{R}^{-1}(H_{X}(x))$,
i.e., for all $U\in D(X)$ such that $x\in m_{R}(U)$. We will prove
that $(x,Z)\in R$. To do so, suppose that $Z\notin R(x)$. From condition
(4) of Definition \ref{t esp mon} we have that there exists $U\in D(X)$
such that $x\in m_{R}(U)$ and $Z\cap U=\emptyset$. Then, by Remark
\ref{prop Hx}, $H_{X}[Z]\cap\beta_{D(X)}(U)=\emptyset$, which is
a contradiction. Therefore, $Z\in R(x)$. \end{proof}

By the following result we get that the dual spaces of monotonic
distributive semilattices are exactly those triples $\langle X,\mathcal{T},R\rangle$,
where $\langle X,\mathcal{T}\rangle$ is a $DS$-space, $R\subseteq X\times\mathcal{S}(X)$,
$m_{R}(U)\in D(X)$, for all $U\in D(X)$, and $\langle X,\mathcal{T},R\rangle$
satisfies any of the equivalent conditions of Theorem \ref{equivalent}.

\begin{lemma}\label{Rupset} Let $\langle X,\mathcal{T},R\rangle$
be a monotonic $DS$-space. Then $R(x)$ is an upset of \linebreak{}
$\langle\mathcal{S}(X),\subseteq\rangle$, i.e., for all $S,Z\in\mathcal{S}(X)$,
and for all $x\in X$, if $S\subseteq Z$ and $S\in R(x)$, then $Z\in R(x)$.
\end{lemma}

\begin{proof} Let $S,Z\in\mathcal{S}(X)$, and $x\in X$, such that
$S\subseteq Z$ and $S\in R(x)$. If $Z\notin R(x)$, then by condition
(4) of Definition \ref{t esp mon}, there exists $U\in D(X)$ such
that $Z\cap U=\emptyset$ and $x\in m_{R}(U)$. But this implies that
$S\cap U=\emptyset$ and $x\in m_{R}(U)$, which is impossible because
$S\in R(x)$. Thus, $R(x)$ is an upset of $\langle\mathcal{S}(X),\subseteq\rangle$.\end{proof}

\begin{theorem} \label{equivalent} Let $\langle X,\mathcal{T}\rangle$
be a $DS$-space. Consider a relation $R\subseteq X\times\mathcal{S}(X)$
such that $m_{R}(U)=\{x\in X:\forall Z\in R(x)[Z\cap U\neq\emptyset]\}\in D(X)$
for all $U\in D(X)$. Then, the following conditions are equivalent,
\begin{enumerate}
\item $R(x)=\bigcap\{L_{U}:x\in m_{R}(U)\text{ and }U\in D(X)\}$ for all
$x\in X$,
\item For all $x\in X$ and for all $Z\in\mathcal{S}(X)$, if $(H_{X}(x),H_{X}[Z])\in R_{m_{R}}$
then $(x,Z)\in R$,
\item $m_{R}(Z^{c})=\bigcup\{m_{R}(U):Z\subseteq U^{c}\text{ and }U\in D(X)\}$
for all $Z\in\mathcal{S}(X)$, and $R(x)$ is an upset of $\langle\mathcal{S}(X),\subseteq\rangle$
for all $x\in X$.
\end{enumerate}
\end{theorem}

\begin{proof} $1.\Rightarrow2$. It was proved in the previous theorem.

$2.\Rightarrow1$. Let $x\in X$. The inclusion $R(x)\subseteq\bigcap\{L_{U}:x\in m_{R}(U)\}$
is clear. Let $Z\in\mathcal{S}(X)$ such that $Z\in\bigcap\{L_{U}:x\in m_{R}(U)\}$.
We will prove that $(H_{X}(x),H_{X}[Z])\in R_{m_{R}}$. Let $U\in D(X)$
such that $x\in m_{R}(U)$. Then, $Z\in L_{U}$, i.e., $Z\cap U\neq\emptyset$.
By Remark \ref{prop Hx}, we have that $H_{X}[Z]\cap\beta_{D(X)}(U)\neq\emptyset$.
Thus, we have that for all $U\in D(X)$ such that $U\in m_{R}^{-1}(H_{X}(x))$,
$H_{X}[Z]\cap\beta_{D(X)}(U)\neq\emptyset$, i.e., $U\notin I_{D(X)}(H_{X}[Z])$.
Therefore, $(H_{X}(x),H_{X}[Z])\in R_{m_{R}}$ and by assumption,
$Z\in R(x)$ and it follows that $\bigcap\{L_{U}:x\in m_{R}(U)\}\subseteq R(x)$.

$1.\Rightarrow3$. Let $x\in m_{R}(Z^{c})$. Then, for all $S\in R(x)$
we have that $S\cap Z^{c}\neq\emptyset$. So, $Z\notin R(x)$. By
assumption, $Z\notin\bigcap\{L_{U}:x\in m_{R}(U)\}$, i.e., there
exists $U\in D(X)$ such that $x\in m_{R}(U)$ and $Z\cap U=\emptyset$.
Thus, $x\in\bigcup\{m_{R}(U):Z\subseteq U^{c}\text{ and }U\in D(X)\}$.
The other inclusion is trivial. The last part is a consequence of
Lemma \ref{Rupset}.

$3.\Rightarrow1$. Let $x\in X$ and $Z\in\bigcap\{L_{U}:x\in m_{R}(U)\text{ and }U\in D(X)\}$.
Suppose that $Z\notin R(x)$. We will see that $x\in m_{R}(Z^{c})$.
On the contrary, suppose that $x\notin m_{R}(Z^{c})$. Then, there
exists $S\in R(x)$ such that $S\cap Z^{c}=\emptyset$. So, $S\subseteq Z$
and by assumption $Z\in R(x)$ which is a contradiction. Thus, $x\in m_{R}(Z^{c})=\bigcup\{m_{R}(U):Z\subseteq U^{c}\text{ and }U\in D(X)\}$,
i.e., there exists $U\in D(X)$ such that $x\in m_{R}(U)$ and $Z\cap U=\emptyset$, a contradiction. Therefore $Z\in R(x)$. The other inclusion
is trivial.\end{proof}

\subsection{Representation of homomorphisms}

In \cite{CelaniTopological} and \cite{CelaniCalomino} it was shown
that there exists a duality between homomorphisms of distributive
semilattices and certain binary relations called meet-relations. It
is also known that $DS$-spaces with meet-relations form a category.
Now, we shall study the representation of homomorphisms of monotonic
distributive semilattices.

Let $S\subseteq X_{1}\times X_{2}$ be a binary relation. Consider
the mapping $h_{S}\colon\mathcal{P}(X_{2})\rightarrow\mathcal{P}(X_{1})$
defined by 
\[
h_{S}(U)=\{x\in X_{1}:S(x)\subseteq U\}.
\] 

A \emph{meet-relation }between two $DS$-spaces $\langle X_{1},\mathcal{T}_{1}\rangle$
and $\langle X_{2},\mathcal{T}_{2}\rangle$ was defined as a subset
$S\subseteq X_{1}\times X_{2}$ satisfying the following conditions:
\begin{enumerate}
\item For every $U\in D(X_{2})$, $h_{S}(U)\in D(X_{1})$, and 
\item $S(x)=\bigcap\{U\in D(X_{2}):S(x)\subseteq U\}$ for all $x\in X_{1}$.
\end{enumerate}

If $S$ is a meet-relation, then $h_{S}$ is a homomorphism between
distributive semilattices.

On the other hand, let $\mathbf{A},\mathbf{B}\in\mathcal{DS}$. Let
$h\colon A\rightarrow B$ be a homomorphism. The binary relation $S_{h}\subseteq X({\mathbf{B}})\times X({\mathbf{A}})$
defined by 
\[
(P,Q)\in S_{h}\text{ iff }h^{-1}[P]\subseteq Q
\]
is a meet-relation, where $h^{-1}[P]=\{a\in A:h(a)\in P\}$.

\begin{definition} \label{cond rel}Let $\langle X_{1},\mathcal{T}_{1},R_{1}\rangle$
and $\langle X_{2},\mathcal{T}_{2},R_{2}\rangle$ be two monotonic
$DS$-spaces. Let us consider a meet-relation $S\subseteq X_{1}\times X_{2}$.
We say that $S$ is a \emph{monotonic meet-relation} if for all $x\in X_{1}$
and every $U\in D(X_{2})$ it satisfies 

\begin{equation}
U^{c}\in R_{2}[S(x)]\text{~iff~}S^{-1}[U^{c}]\in R_{1}(x)\label{eq:homo}
\end{equation}
 where $R_{2}[S(x)]=\{Z\in\mathcal{S}(X_{2}):\exists y\in S(x)\ [(y,Z)\in R_{2}]\}$.

\end{definition}

\begin{remark}Note that if $S\subseteq X_{1}\times X_{2}$ is a meet-relation
between two $DS$-spaces $\langle X_{1},\mathcal{T}_{1}\rangle$ and
$\langle X_{2},\mathcal{T}_{2}\rangle$, then $S^{-1}[U^{c}]=h_{S}(U)^{c}\in\mathcal{S}(X_{1})$.

\end{remark}

\begin{proposition} \label{equivalence homo}The condition (\ref{eq:homo})
is equivalent to the condition 
\[
h_{S}(m_{R_{2}}(U))=m_{R_{1}}(h_{S}(U))
\]
for all $U\in D_{\mathcal{K}_{2}}(X_{2})$, i.e., the mapping $h_{S}\colon D(X_{2})\rightarrow D(X_{1})$
is a homomorphism of monotonic distributive semilattices.\end{proposition}

\begin{proof}

$\Rightarrow)$ Suppose that for all $x\in X_{1}$ and every $U\in D(X_{2})$, $U^{c}\in R_{2}[S(x)]$ if and only if $S^{-1}[U^{c}]\in R_{1}(x)$.
Let $x\in h_{S}(m_{R_{2}}(U))$, i.e., $S(x)\subseteq m_{R_{2}}(U)$.
Then, for all $y\in S(x)$ we have that $y\in m_{R_{2}}(U)$. So,
for all $y\in S(x)$ and for all $Z\in R_{2}(y)$ we have that $Z\cap U\neq\emptyset.$
Then, for all $y\in S(x)$, $U^{c}\notin R_{2}(y)$. Thus, $U^{c}\notin R_{2}[S(x)]$.
By hypothesis, $S^{-1}[U^{c}]\notin R_{1}(x)$. Therefore, $x\in m_{R_{1}}(S^{-1}[U^{c}]^{c})=m_{R_{1}}(h_{S}(U))$.
The other inclusion is obtained reverting the implications. 

$\Leftarrow)$ Suppose that $h_{S}$ is a homomorphism. Let $U^{c}\in R_{2}[S(x)]$.
Then, there exists $y\in S(x)$ such that $U^{c}\in R_{2}(y)$. So,
$y\notin m_{R_{2}}(U)$. Thus, $S(x)\nsubseteq m_{R_{2}}(U)$, i.e.,
$x\notin h_{S}(m_{R_{2}}(U))$. By hypothesis, $x\notin m_{R_{1}}(h_{S}(U))$,
i.e., there exists $Z\in R_{1}(x)$ such that $Z\cap h_{S}(U)=\emptyset$.
We have that $Z\subseteq h_{S}(U)^{c}$ and since $h_{S}(U)^{c}\in\mathcal{S}(X)$
we have that $S^{-1}[U^{c}]=h_{S}(U)^{c}\in R_{1}(x)$. The other
implication is obtained similarly.\end{proof}

Now, we will study the composition of monotonic meet-relations. Let
$X_{1}$, $X_{2}$ and $X_{3}$ be sets. Let us consider two
relations $S_{1}\subseteq X_{1}\times X_{2}$ and $S_{2}\subseteq X_{2}\times X_{3}$.
Then, the composition of $S_{1}$ and $S_{2}$ is the relation $S_{2}\circ S_{\text{1}}\subseteq X_{1}\times X_{3}$
defined by 
\[
S_{2}\circ S_{\text{1}}=\{(x,z)\in X_{1}\times X_{3}:\exists y\in X_{2}[(x,y)\in S_{1}\text{ and }(y,z)\in S_{2}\}.
\]

\begin{proposition} Let $\langle X_{1},\mathcal{T}_{1},R_{1}\rangle$,
$\langle X_{2},\mathcal{T}_{2},R_{2}\rangle$ and $\langle X_{3},\mathcal{T}_{3},R_{3}\rangle$
be three monotonic $DS$-spaces. Let us consider two monotonic meet-relations
$S_{1}\subseteq X_{1}\times X_{2}$ and $S_{2}\subseteq X_{2}\times X_{3}$.
Then, $S_{3}=S_{2}\circ S_{1}\subseteq X_{1}\times X_{3}$ is a monotonic
meet-relation. \end{proposition}

\begin{proof} It follows from the fact that $h_{S_{3}}(U)=h_{S_{2}\circ S_{1}}(U)=h_{S_{1}}\circ h_{S_{2}}(U)$
for all $U\in D_{\mathcal{K}_{3}}(X_{3})$, definition \ref{cond rel}
and proposition \ref{equivalence homo}. \end{proof}

\begin{proposition} Let $\langle X,\mathcal{T},R\rangle$ be a monotonic
$DS$-space. The specialization dual order $\leq\subseteq X\times X$
is a monotonic meet-relation. \end{proposition}

\begin{proof} $\Rightarrow)$ Let $U\in D(X)$ and suppose that $U^{c}\in R([x))$.
Then, there exists $y\geq x$ such that $U^{c}\in R(y)$ and since
$R(y)\subseteq R(x)$, we have that $U^{c}\in R(x)$. As $U^{c}$
is a downset, $\leq^{-1}[U^{c}]=U^{c}$.

The other implication is trivial. \end{proof}

So, monotonic $DS$-spaces with monotonic meet-relations form a category
where the identity arrow is the specialization dual order. We will
denote this category by $\mathcal{MDSR}$. 

\begin{proposition}\label{homo cond} Let $\langle\mathbf{A},m_{\mathbf{A}}\rangle$,
$\langle\mathbf{B},m_{\mathbf{B}}\rangle\in\mathcal{MDS}$.
\begin{enumerate}
\item Let $h\colon A\rightarrow B$ be a monotonic homomorphism. Then, the
meet-relation $S_{h}$ satisfies condition (\ref{eq:homo}).
\item Let $h\colon A\rightarrow B$ be a homomorphism and suppose that the
meet-relation $S_{h}$ satisfies condition (\ref{eq:homo}). Then,
$h$ is monotonic. 
\end{enumerate}
\end{proposition}

\begin{proof} 1. Suppose that $h$ is a monotonic homomorphism. So,
it is easy to see that $h_{S_{h}}(\beta_{\mathbf{A}}(a))=\beta_{\mathbf{B}}(h(a))$
for all $a\in A$. Then, we have 
\begin{align*}
h_{S_{h}}(m_{R_{m_{\mathbf{A}}}}\beta_{\mathbf{A}}(a)) & =h_{S_{h}}(\beta_{\mathbf{A}}(m_{\mathbf{A}}a))=\beta_{\mathbf{B}}(h(m_{\mathbf{A}}a))\\
 & =\beta_{\mathbf{B}}(m_{\mathbf{B}}h(a))=m_{R_{m_{\mathbf{B}}}}(\beta_{\mathbf{B}}(h(a)))\\
 & =m_{R_{m_{\mathbf{B}}}}(h_{S_{h}}(\beta_{\mathbf{A}}(a)))
\end{align*}
for all $a\in A$.

2. Suppose that $h$ is a homomorphism and that $S_{h}$ satisfies
condition (\ref{eq:homo}). Then, $h_{S_{h}}(\beta_{\mathbf{A}}(a))=\beta_{\mathbf{B}}(h(a))$
for all $a\in A$. So, we have 
\begin{align*}
\beta_{\mathbf{B}}(h(m_{\mathbf{A}}a)) & =h_{S_{h}}(\beta_{\mathbf{A}}(m_{\mathbf{A}}a))=h_{S_{h}}(m_{R_{m_{\mathbf{A}}}}\beta_{\mathbf{A}}(a))\\
 & =m_{R_{m_{\mathbf{B}}}}(h_{S_{h}}(\beta_{\mathbf{A}}(a)))=m_{R_{m_{\mathbf{B}}}}(\beta_{\mathbf{B}}(h(a)))\\
 & =\beta_{\mathbf{B}}(m_{\mathbf{B}}h(a))
\end{align*}
and since $\beta_{\mathbf{B}}$ is an injective function, we get that
$h(m_{\mathbf{A}}a)=m_{\mathbf{B}}h(a)$ for all $a\in A$. \end{proof}

From Theorem \ref{HX} and Proposition \ref{equivalence homo}, we
conclude that the functor $\mathbb{D}:\mathcal{MDSR}\rightarrow\mathcal{MDSH}$
defined by 
\begin{enumerate}
\item $\mathbb{D}(X)=\langle D(X),m_{R}\rangle$ if $\langle X,\mathcal{T},R\rangle$
is a $DS$-space, 
\item $\mathbb{D}(S)=h_{S}$ if $S$ is a monotonic meet-relation
\end{enumerate}
is a contravariant functor. By Theorem \ref{reo mon}, Corollary \ref{rep cor}
and Proposition \ref{homo cond}, we conclude that the functor $\mathbb{X}:\mathcal{MDSH}\rightarrow \mathcal{MDSR}$
defined by 
\begin{enumerate}
\item $\mathbb{X}(\mathbf{A})=\langle X(\mathbf{A});\mathcal{T}_{A},R_{m}\rangle$
if $\langle\mathbf{A},m\rangle$ is a monotonic distributive semilattice, 
\item $\mathbb{X}(h)=S_{h}$ if $h$ is homomorphism of monotonic distributive
semilattices
\end{enumerate}
is a contravariant functor. Therefore, we give the following result.

\begin{corollary} The categories $\mathcal{MDSH}$ and $\mathcal{MDSR}$
are dually equivalent. \end{corollary}

\section{Applications of the duality}

In this section we consider some applications of the duality. We will
consider some important subclasses and show how our new duality extends
the one developed in \cite{Celaniboole} for Boolean algebras.

\subsection{Additional conditions}

Now we will see how some additional conditions affect the relations
associated to the monotonic operator. 

The following formulas are $\pi$- and $\sigma$-canonical, i.e.,
their validity is preserved under taking $\pi$- and $\sigma$-canonical
extensions.

\begin{proposition} Let $\langle A,m\rangle\in\mathcal{MDS}$. Then,
\begin{enumerate}
\item $m1=1$	 iff	 $\forall P\in X(\mathbf{A})\ [\emptyset\notin R_{m}(P)]$
	iff	 $\forall P\in X(\mathbf{A})\ [X(\mathbf{A})\in G_{m}(P)]$;
\item $m0=0$ 	iff 	$\forall P\in X(\mathbf{A})\ [X(\mathbf{A})\in R_{m}(P)]$
	iff	 $\forall P\in X(\mathbf{A})\ [\emptyset\notin G_{m}(P)]$;
\item $\forall a\in A\ [ma\leq a]$ 	iff	 $\forall P\in X(\mathbf{A})\ [\alpha(P^{c})=(P]\in R_{m}(P)]$
	iff\\
 $\forall P\in X(\mathbf{A})\forall Y\in G_{m}(P)\ [P\in Y]$;
\item $\forall a\in A\ [a\leq ma]$ 	iff	$\forall P\in X(\mathbf{A})\forall Z\in R_{m}(P)\ [P\in Z]$
	iff \\
$\forall P\in X(\mathbf{A})\ [\hat{P}=[P)\in G_{m}(P)]$.
\end{enumerate}

\end{proposition}

\begin{proof}

1. Suppose that $m1=1$ and suppose that there exists $P\in X(\mathbf{A})$
such that $\emptyset\in R_{m}(P)$. Then, $m1=1\in P$ and $m^{-1}(P)\cap I_{\mathbf{A}}(\emptyset)=m^{-1}(P)\cap A=\emptyset$
and it follows that $m^{-1}(P)=\emptyset$ which is a contradiction. 

Now, suppose that for all $P\in X(\mathbf{A})$ we have $\emptyset\notin R_{m}(P)$
and suppose that there exists $P\in X(\mathbf{A})$ such that $X(\mathbf{A})\notin G_{m}(P)$.
Then, $F_{X(\mathbf{A})}=\{1\}\nsubseteq m^{-1}(P)$ , i.e., $1\notin m^{-1}(P)$.
Since $P$ is an upset, $m^{-1}(P)=\emptyset$. So, we have that $m^{-1}(P)\cap A=m^{-1}(P)\cap I_{\mathbf{A}}(\emptyset)=\emptyset$
and by definition $\emptyset\in R_{m}(P)$ which is a contradiction.

Suppose that for all $P\in X(\mathbf{A})$ we have $X(\mathbf{A})\in G_{m}(P)$
and suppose that $m1\neq1$. Then, there exists $P\in X(\mathbf{A})$
such that $m1\notin P$. So, we have that $F_{X(\mathbf{A})}=\{1\}\nsubseteq m^{-1}(P)$
which is a contradiction.

2. The proof is similar to 1.

3. Suppose that $ma\leq a$ for all $a\in A$ and that there exists
$P\in X(\mathbf{A})$ such that $(P]\notin R_{m}(P)$. Then, $m^{-1}(P)\cap P^{c}\neq\emptyset$.
So, there exists $a\in A$ such that $ma\in P$ and $a\notin P$,
which is a contradiction. 

Now, suppose that for all $P\in X(\mathbf{A})$ we have $(P]\in R_{m}(P)$
and suppose that there exists $P\in X(\mathbf{A})$ and there exists
$Y\in G_{m}(P)$ such that $P\notin Y$. Then, $F_{Y}\subseteq m^{-1}(P)$
and there exists $a\in F_{Y}$ such that $a\notin P$. So, $a\in m^{-1}(P)\cap P^{c}$
and it follows that $(P]\notin R_{m}(P)$ which is a contradiction.

Suppose that for all $P\in X(\mathbf{A})$ and for all $Y\in G_{m}(P)$
we have $P\in Y$ and suppose that $ma\nleq a$. Then, there exists
$P\in X(\mathbf{A})$ such that $ma\in P$ and $a\notin P$. So, we
have that $[a)\subseteq m^{-1}(P)$ but $P\notin\widehat{[a)}$ which
is a contradiction.

4. The proof is similar to 3.\end{proof}

Now, we will characterize the dual spaces of monotonic distributive
meet-semilattices satisfying condition ($\mathbf{4}_{\square}$)
$ma\leq m^{2}a$ or condition ($\mathbf{4}_{\Diamond}$) $m^{2}a\leq ma$ for
every element $a$. We will see that condition $\mathbf{4}_{\square}$
is $\sigma$-canonical and that condition $\boldsymbol{4}_{\Diamond}$
is $\pi$-canonical.

Let $\langle X,\mathcal{T},R\rangle$ be a monotonic $DS$-space.
For any $U\in\mathrm{Up}(X)$, we define the operator $m_{R}^{2}\colon\mathrm{Up}(X)\rightarrow\mathrm{Up}(X)$
by $m_{R}^{2}(U)=m_{R}(m_{R}(U))$.

\begin{remark}Let $\langle X,\mathcal{T},R\rangle$ be a monotonic
$DS$-space and let $Z\in\mathcal{S}(X)$. Then recall that $m_{R}(Z^{c})^{c}=\bigcap\{m_{R}(U)^{c}:U\in D(X)\text{ and }Z\subseteq U^{c}\}$.

\end{remark}

\begin{proposition}Let $\langle\mathbf{A},m\rangle\in\mathcal{MDS}$ such that $m^{2}a\leq ma$ for all $a\in A$. Then, $m_{R_{m}}^{2}(U)\subseteq m_{R_{m}}(U)$
for all $U\in\mathrm{Up}(X(\mathbf{A}))$, i.e., $\boldsymbol{4}_{\Diamond}$
is $\pi$-canonical.

\end{proposition}

\begin{proof}Let $A\in\mathcal{MDS}$ such that $m^{2}a\leq ma$
for all $a\in A$. Then, for all $U\in D(X(\mathbf{A}))$ we have
that $m_{R_{m}}^{2}(U)\subseteq m_{R_{m}}(U)$. First, we will see
that $m_{R_{m}}^{2}(Z^{c})\subseteq m_{R_{m}}(Z^{c})$ for all $Z\in\mathcal{S}(X(\mathbf{A}))$.
Let $P\in m_{R_{m}}^{2}(Z^{c})$. So, we get that $m_{R_{m}}(Z^{c})^{c}\notin R_{m}(P)$.
Suppose that $P\notin m_{R_{m}}(Z^{c})$. Then, we get that $Z\in R_{m}(P)$.
Since $R_{m}(P)=\bigcap\{L_{U}:U\in D(X)\text{ and }P\in m_{R_{m}}(U)\}$,
there exists $U\in D(X)$ such that $P\in m_{R_{m}}(U)$ and $m_{R_{m}}(Z^{c})^{c}\cap U=\emptyset$.
By the previous remark, there exists $V\in D(X)$ such that $Z\subseteq V^{c}$
and $m_{R_{m}}(V)^{c}\cap U=\emptyset$. Thus, $U\subseteq m_{R_{m}}(V)$
and by hypothesis $P\in m_{R_{m}}(U)\subseteq m_{R_{m}}^{2}(V)\subseteq m_{R_{m}}(V)$.
Since $P\in m_{R_{m}}(V)$ and $Z\in R_{m}(P)$ we get that $Z\cap V\neq\emptyset$
which is a contradiction. 

Now, we will see that $m_{R_{m}}^{2}(U)\subseteq m_{R_{m}}(U)$ for
all $U\in\mathrm{Up}(X(\mathbf{A}))$. Let $P\in m_{R_{m}}^{2}(U)$
and suppose that $P\notin m_{R_{m}}(U)$. Then, there exists $Z\in R_{m}(P)$
such that $Z\cap U=\emptyset$. So, $U\subseteq Z^{c}$ and we get
that $m_{R_{m}}(U)\subseteq m_{R_{m}}(Z^{c})$. Thus, $m_{R_{m}}(U)\cap m_{R_{m}}(Z^{c})^{c}=\emptyset$
and since $P\in m_{R_{m}}^{2}(U)$ we get that $m_{R_{m}}(Z^{c})^{c}\notin R_{m}(P)$.
Therefore $P\in m_{R_{m}}^{2}(Z^{c})\subseteq m_{R_{m}}(Z^{c})$ which
is a contradiction because $Z\in R_{m}(P)$.\end{proof}

Let $\langle X,\mathcal{T},R\rangle$ be a monotonic $DS$-space.
We will define a relation $\bar{R}\subseteq\mathcal{S}(X)\times\mathcal{S}(X)$
by
\[
(S,Z)\in\bar{R}\Leftrightarrow\forall x\in S\ (x,Z)\in R.
\]

We define $R^{2}\subseteq X\times\mathcal{S}(X)$ as follows 
\[
(x,Z)\in R^{2}\Leftrightarrow\exists S\in\mathcal{S}(X)\text{ such that }(x,S)\in R\text{ and }(S,Z)\in\bar{R}.
\]

\begin{definition}Let $\langle X,\mathcal{T},R\rangle$ be a monotonic
$DS$-space. The relation $R$ is \emph{transitive} if and only if
for all $x\in X$ and for all $Z\in\mathcal{S}(X)$ if $(x,Z)\in R^{2}$
then $(x,Z)\in R$.

\end{definition}

\begin{definition}Let $\langle X,\mathcal{T},R\rangle$ be a monotonic
$DS$-space. The relation $R$ is \emph{weakly dense} if and only
if for all $x\in X$ and for all $Z\in\mathcal{S}(X)$ if $(x,Z)\in R$
then $(x,Z)\in R^{2}$.

\end{definition}

\begin{corollary}Let $\langle A,m\rangle\in\mathcal{MDS}$. Then,
$I_{\mathbf{A}}(m_{R_{m}}(\alpha(I)^{c})^{c})=(m(I)]$ where $m(I)=\{ma:a\in I\}$.

\end{corollary}

\begin{lemma} \label{Lemma ideal}Let $A\in\mathcal{MDS}$. Then
for all $P\in X(\mathbf{A})$ and $I\in\mathrm{Id}(\mathbf{A})$,
$(P,\alpha(I))\in R_{m}^{2}$ if and only if $I\subseteq\{a\in A:m^{2}a\in P^{c}\}$.

\end{lemma}

\begin{proof}$\Rightarrow)$ Suppose that $(P,\alpha(I))\in R_{m}^{2}$
and let $a\in I$. Suppose that $m^{2}a\in P$. Then, there exists
$J\in\mathrm{Id}(\mathbf{A})$ such that $(P,\alpha(J))\in R_{m}$
and $(\alpha(J),\alpha(I))\in\bar{R}_{m}$. So, $m^{-1}(P)\cap J=\emptyset$
and $ma\notin J$. Thus, there exists $Q\in\alpha(J)$ such
that $ma\in Q$. Therefore, $(Q,\alpha(I))\in R_{m}$ and $a\notin I$
which is a contradiction. 

$\Leftarrow)$ Suppose that $I\subseteq\{a\in A:m^{2}a\in P^{c}\}$.
We will prove that $m_{R_{m}}(\alpha(I)^{c})^{c}\in R_{m}(P)$. Since
$I_{\mathbf{A}}(m_{R_{m}}(\alpha(I)^{c})^{c})=(m(I)]$, suppose that
there exists $a\in A$ such that $a\in m^{-1}(P)\cap(m(I)]$. So,
$ma\in P$ and there exists $b\in I$ such that $a\leq mb$. Then,
$ma\leq m^{2}b$ and we get that $m^{2}b\in P\cap P^{c}$ which is
a contradiction. Therefore, $m_{R_{m}}(\alpha(I)^{c})^{c}\in R_{m}(P)$
and $(m_{R_{m}}(\alpha(I)^{c})^{c},\alpha(I))\in\bar{R}_{m}$, i.e.,
$(P,\alpha(I))\in R_{m}^{2}$.\end{proof}

\begin{proposition} Let $A\in\mathcal{MDS}$. Then $ma\leq m^{2}a$
for all $a\in A$ if and only if $R_{m}$ is transitive. 

\end{proposition}

\begin{proof}$\Rightarrow)$ Suppose that $ma\leq m^{2}a$ for all
$a\in A$ and that $(P,Z)\in R_{m}^{2}$. Then, $I_{\mathbf{A}}(Z)\subseteq\{a\in A:m^{2}a\in P^{c}\}$.
Suppose that $m^{-1}(P)\cap I_{\mathbf{A}}(Z)\neq\emptyset$. We get
that there exists $a\in A$ such that $ma\in P$ and $a\in I_{\mathbf{A}}(Z)$.
Therefore $m^{2}a\notin P$ and $m^{2}a\in P$ which is a contradiction.
Since $m^{-1}(P)\cap I_{\mathbf{A}}(Z)=\emptyset$ we get that $(P,Z)\in R_{m}$.

$\Leftarrow)$ Suppose that $R_{m}$ is transitive and suppose that
there exists $a\in A$ such that $ma\nleq m^{2}a$. Then, there exists
$P\in X(\mathbf{A})$ such that $ma\in P$ and $m^{2}a\notin P$.
So, $(a]\subseteq\{a\in A:m^{2}a\in P^{c}\}$ and by Lemma \ref{Lemma ideal},
$(P,\alpha(a))\in R_{m}^{2}$. We get that $(P,\alpha(a))\in R_{m}$,
i.e., $ma\notin P$, which is a contradiction.\end{proof}

\begin{proposition} Let $A\in\mathcal{MDS}$. Then $m^{2}a\leq ma$
for all $a\in A$ if and only if $R_{m}$ is weakly dense.

\end{proposition}

\begin{proof}$\Rightarrow)$ Suppose that $m^{2}a\leq ma$ for all
$a\in A$ and that $(P,Z)\in R_{m}$. We will prove that $I_{\mathbf{A}}(Z)\subseteq\{a\in A:m^{2}a\in P^{c}\}$.
Let $a\in I_{\mathbf{A}}(Z)$. Since $m^{-1}(P)\cap I_{\mathbf{A}}(Z)=\emptyset$
we get that $ma\notin P$ and therefore $m^{2}a\notin P$. By Lemma
\ref{Lemma ideal}, $(P,Z)\in R_{m}^{2}$.

$\Leftarrow)$ Suppose that $R_{m}$ is weakly dense. Suppose that
there exists $a\in A$ such that $m^{2}a\nleq ma$. Then, there exists
$P\in X(\mathbf{A})$ such that $m^{2}a\in P$ and $ma\notin P$.
Then, $(P,\alpha(a))\in R_{m}$. So, $(P,\alpha(a))\in R_{m}^{2}$
and, by Lemma \ref{Lemma ideal}, $(a]\subseteq\{a\in A:m^{2}a\in P^{c}\}$,
i.e., $m^{2}a\notin P$, which is a contradiction.\end{proof}

For the sake of completeness we add the corresponding definitions
and theorems for $\mathcal{C}$-monotonic $DS$-spaces.

Let $\langle X,\mathcal{T},G\rangle$ be a $\mathcal{C}$-monotonic
$DS$-space. For any $U\in\mathrm{Up}(X)$ we define the operator
$\mathbf{m}_{G}^{2}\colon\mathrm{Up}(X)\rightarrow\mathrm{Up}(X)$ by $\mathbf{m}_{G}^{2}(U)=\mathbf{m}_{G}(\mathbf{m}_{G}(U))$.

\begin{remark}Let $\langle X,\mathcal{T},G\rangle$ be a $\mathcal{C}$-monotonic
$DS$-space and $Y\in\mathcal{C}(X)$. Recall that $\mathbf{m}_{G}(Y)=\bigcap\{\mathbf{m}_{G}(U):U\in D(X)\text{ and }Y\subseteq U\}$.

\end{remark}

\begin{proposition}Let $\langle\mathbf{A},m\rangle\in\mathcal{MDS}$
such that $ma\leq m^{2}a$ for all $a\in A$. Then, $\mathbf{m}_{G_{m}}(U)\subseteq \mathbf{m}_{G_{m}}^{2}(U)$
for all $U\in\mathrm{Up}(X(\mathbf{A}))$, i.e., $\boldsymbol{4}_{\square}$
is $\sigma$-canonical.

\end{proposition}

\begin{proof}Let $A\in\mathcal{MDS}$ such that $ma\leq m^{2}a$
for all $a\in A$. Then, for all $U\in D(X(\mathbf{A}))$ we have
that $\mathbf{m}_{G_{m}}(U)\subseteq \mathbf{m}_{G_{m}}^{2}(U)$. First, we will see
that $\mathbf{m}_{G_{m}}(Y)\subseteq \mathbf{m}_{G_{m}}^{2}(Y)$ for all $Y\in\mathcal{C}(X(\mathbf{A}))$.
Let $P\in \mathbf{m}_{G_{m}}(Y)$. So, we get that $Y\in G_{m}(P)$. Suppose
that $P\notin \mathbf{m}_{G_{m}}^{2}(Y)$. Then, we get that $\mathbf{m}_{G_{m}}(Y)\notin G_{m}(P)$.
Since $G_{m}(P)=\bigcap\{(D_{U})^{c}:U\in D(X)\text{ and }P\notin \mathbf{m}_{G_{m}}(U)\}$,
there exists $U\in D(X)$ such that $P\notin \mathbf{m}_{G_{m}}(U)$ and $\mathbf{m}_{G_{m}}(Y)\subseteq U$.
By the previous remark, there exists $V\in D(X)$ such that $Y\subseteq V$
and $\mathbf{m}_{G_{m}}(Y)\subseteq \mathbf{m}_{G_{m}}(V)\subseteq U$. Thus, $P\in \mathbf{m}_{G_{m}}(V)$
and by hypothesis $P\in \mathbf{m}_{G_{m}}^{2}(V)\subseteq \mathbf{m}_{G_{m}}(U)$ which
is a contradiction. 

Now, we will see that $\mathbf{m}_{G_{m}}(U)\subseteq \mathbf{m}_{G_{m}}^{2}(U)$ for
all $U\in\mathrm{Up}(X(\mathbf{A}))$. Let $P\in \mathbf{m}_{G_{m}}(U)$. Then,
there exists $Y\in G_{m}(P)$ such that $Y\subseteq U$. So, $\mathbf{m}_{G_{m}}(Y)\subseteq \mathbf{m}_{G_{m}}(U)$
and we get that $P\in \mathbf{m}_{G_{m}}(Y)\subseteq \mathbf{m}_{G_{m}}^{2}(Y)\subseteq \mathbf{m}_{G_{m}}^{2}(U)$.
Thus, $P\in \mathbf{m}_{G_{m}}^{2}(U)$.\end{proof}

Let $\langle X,\mathcal{T},G\rangle$ be a $\mathcal{C}$-monotonic
$DS$-space. We will define a relation $\bar{G}\subseteq\mathcal{C}(X)\times\mathcal{C}(X)$
by
\[
(Y,C)\in\bar{G}\Leftrightarrow\forall x\in Y\ (x,C)\in G.
\]

We define $G^{2}\subseteq X\times\mathcal{C}(X)$ as follows 
\[
(x,Y)\in G^{2}\Leftrightarrow\exists C\in\mathcal{C}(X)\text{ such that }(x,C)\in G\text{ and }(C,Y)\in\bar{G}.
\]

\begin{definition}Let $\langle X,\mathcal{T},G\rangle$ be a $\mathcal{C}$-
monotonic $DS$-space. The relation $G$ is \emph{transitive} if and
only if for all $x\in X$ and for all $Y\in\mathcal{C}(X)$ if $(x,Y)\in G^{2}$
then $(x,Y)\in G$.

\end{definition}

\begin{definition}Let $\langle X,\mathcal{T},G\rangle$ be a a $\mathcal{C}$-
monotonic $DS$-space. The relation $G$ is \emph{weakly dense} if
and only if for all $x\in X$ and for all $Y\in\mathcal{C}(X)$ if
$(x,Y)\in G$ then $(x,Y)\in G^{2}$.

\end{definition}

\begin{lemma}Let $A\in\mathcal{MDS}$ and $F\in\mathrm{Fi}(\mathbf{A})$.
Then, $F_{\mathbf{m}_{G_{m}}(\hat{F})}=[m(F))$ where $m(F)=\{ma:a\in F\}$.

\end{lemma}

\begin{lemma} Let $A\in\mathcal{MDS}$. Then for all $P\in X(\mathbf{A})$
and $F\in\mathrm{Fi}(\mathbf{A})$, $(P,\hat{F})\in G_{m}^{2}\Leftrightarrow F\subseteq\{a\in A:m^{2}a\in P\}$.

\end{lemma}

\begin{proposition} Let $A\in\mathcal{MDS}$. Then $ma\leq m^{2}a$
for all $a\in A$ if and only if $G_{m}$ is weakly dense.

\end{proposition}

\begin{proposition} Let $A\in\mathcal{MDS}$. Then $m^{2}a\leq ma$ for all $a\in A$ if and only if $G_{m}$ is transitive.

\end{proposition}

\subsection{Modal distributive semilattices}

\label{subsection: Modal}

In this section we consider distributive semilattices endowed with
a normal modal operator, i.e., a function that preserves the greatest element and finite meets. 

\begin{definition}A \emph{modal distributive semilattice} is an algebra
$\langle\mathbf{A},m\rangle$ where $\mathbf{A}$ is a distributive
semilattice and $m\colon A\rightarrow A$ is an operator that verifies
the following conditions:
\begin{enumerate}
\item $m1=1$,
\item $m(a\wedge b)=ma\wedge mb$ for all $a,b\in A$. 
\end{enumerate}
\end{definition}

It is clear that $m$ is a homomorphism and that a modal distributive
semilattice is a monotonic distributive semilattice. Thus, a modal
operator $m$ could be
represented by means of an adequate multirelation defined in
the dual space and by a meet-relation. Now we are
going to identify what additional conditions must satisfy this multirelation.

\begin{remark} Let $\langle X,\mathcal{T},R\rangle$ be a monotonic
$DS$-space. Since $(x]=\bigcap\{U\in\mathcal{KO}(X):x\in U\}$ and
$\mathcal{KO}(X)$ is a basis, we have that $(x]\in\mathcal{S}(X)$
for each $x\in X$. \end{remark}

\begin{remark} Given a modal distributive semilattice $\langle\mathbf{A},m\rangle$,
we note that $m^{-1}(F)\in\Fi({\mathbf{A}})$ for all $F\in\Fi({\mathbf{A}})$.
We also note that $I_{\mathbf{A}}((Q])=Q^{c}$ for all $Q\in X({\mathbf{A}})$.
\end{remark}

\begin{definition} A monotonic $DS$-space $\langle X,\mathcal{T},R\rangle$
is called \emph{normal} if for any $x\in X$ and for every $Z\in\mathcal{S}(X)$
such that $Z\in R(x)$ there exists $z\in Z$ such that $(z]\in R(x)$.
\end{definition}

Note that in every normal monotonic $DS$-space $\langle X,\mathcal{T},R\rangle$ for all $x\in X$, $\emptyset\notin R(x)$.

\begin{proposition} Let $\langle\mathbf{A},m\rangle$ be a monotonic
distributive semilattice. Then $\langle\mathbf{A},m\rangle$ is a
modal distributive semilattice iff $\langle X({\mathbf{A}}),\mathcal{T}_{\mathbf{A}},R_{m}\rangle$
is a normal monotonic $DS$-space. \end{proposition}

\begin{proof} $\Rightarrow)$ Let $(P,\alpha(I))\in R_{m}$. Then,
$m^{-1}(P)\cap I=\emptyset$. Since $\langle\mathbf{A},m\rangle$
is a modal distributive semilattice, we have that $m^{-1}(P)\in\Fi({\mathbf{A}})$.
So, there exists $Q\in X({\mathbf{A}})$ such that $m^{-1}(P)\subseteq Q$
and $Q\cap I=\emptyset$. Thus, $Q\in\alpha(I)$ and it is easy to
see that $m^{-1}(P)\cap Q^{c}=\emptyset$ and therefore, $(P,(Q])\in R_{m}$.

$\Leftarrow)$ Let $\langle X({\mathbf{A}}),\mathcal{T}_{\mathbf{A}},R_{m}\rangle$
be a normal monotonic $DS$-space. Suppose that there exist $a,b\in A$
such that $ma\wedge mb\nleq m(a\wedge b)$. Then, there exists $P\in X({\mathbf{A}})$
such that $ma\wedge mb\in P$ but $m(a\wedge b)\notin P$. Note that
$ma\wedge mb\leq ma\in P$ and $ma\wedge mb\leq mb\in P$. So, we
have that $(P,\alpha(a\wedge b))\in R_{m}$ and since $\langle X({\mathbf{A}}),\mathcal{T}_{\mathbf{A}},R_{m}\rangle$
is a normal monotonic $DS$-space, there exists $Q\in\alpha(a\wedge b)$
such that $(P,(Q])\in R_{m}$. From $I_{\mathbf{A}}((Q])=Q^{c}$ we
get that $m^{-1}(P)\cap Q^{c}=\emptyset$. Thus, $a,b\in m^{-1}(P)\subseteq Q$
and, since $Q$ is a filter, we have that $a\wedge b\in Q$. Hence,
$Q\cap(a\wedge b]\neq\emptyset$ which contradicts the fact that $Q\in\alpha(a\wedge b)$.
Therefore $\langle\mathbf{A},m\rangle$ is a modal distributive semilattice.\end{proof}

Since a normal modal operator is also a homomorphism of meet semilattices,
we can also interpret it through a meet-relation in the dual space.
We will show the relationship between the multirelation and the meet-relation
associated to the same operator. 

Let $\langle X,\mathcal{T}\rangle$ be a $DS$-space and let $S\subseteq X\times X$
be a meet-relation. Let $\mathrm{m}_{S}\colon\mathrm{Up}(X)\rightarrow\mathrm{Up}(X)$
be the operator defined by 
\[
\mathrm{m}_{S}(U)=\{x\in X:S(x)\subseteq U\}
\]
where $S(x)=\{y\in X:(x,y)\in S\}$. We define a multirelation $R_{S}\subseteq X\times\mathcal{S}(X)$
by 
\[
(x,Z)\in R_{S}\Leftrightarrow S(x)\cap Z\neq\emptyset.
\]
On the other hand, let $\langle X,\mathcal{T},R\rangle$ be a normal
monotonic $DS$-space. We define a relation $S_{R}\subseteq X\times X$
by 
\[
(x,z)\in S_{R}\Leftrightarrow(x,(z])\in R.
\]

\begin{proposition}
\begin{enumerate}
\item Let $\langle X,\mathcal{T},R\rangle$ be a normal monotonic $DS$-space.
Then, the structure $\langle X,\mathcal{T},S_{R}\rangle$ is a $DS$-space with a meet-relation $S_R$ such
that $m_{R}(U)=\mathrm{m}_{S_{R}}(U)$ for all $U\in\mathrm{Up}(X$) and $R=R_{S_{R}}$.
\item Let $\langle X,\mathcal{T},S\rangle$ be a $DS$-space with a meet-relation
$S\subseteq X\times X$ . Then, the structure $\langle X,\mathcal{T},R_{S}\rangle$
is a normal monotonic $DS$-space such that $m_{R_{S}}(U)=\mathrm{m}_{S}(U)$
for all $U\in\mathrm{Up}(X)$, and $S=S_{R_{S}}$.
\end{enumerate}
\end{proposition}

\begin{proof} 1. Let $U\in\mathrm{Up}(X)$. We will prove that $m_{R}(U)=\mathrm{m}_{S_{R}}(U)$.
Let $x\in m_{R}(U)$ and $z\in S_{R}(x)$. Then, $(x,(z])\in R$
and we get that $(z]\cap U\neq\emptyset$. Since $U$ is an upset,
we have that $z\in U$. Thus, $S_{R}(x)\subseteq U$ and $x\in \mathrm{m}_{S_{R}}(U)$.
Let $x\in \mathrm{m}_{S_{R}}(U)$ and $Z\in R(x)$. Then, there exists
$z\in Z$ such that $(z]\in R(x)$. By definition $(x,z)\in S_{R}(x)$.
So, $z\in U$ and we get that $Z\cap U\neq\emptyset$. Thus, $x\in m_{R}(U)$. 

So, we have that $m_{R}(U)=\mathrm{m}_{S_{R}}(U)\in D(X)$ for all $U\in D(X)$.
We will see that $S_{R}(x)=\bigcap\{U\in D(X):S_{R}(x)\subseteq U\}$
for all $x\in X$. Let $x,z\in X$ such that $z\in\bigcap\{U\in D(X):S_{R}(x)\subseteq U\}$.
Then, $z\in U$ for all $U\in D(X)$ such that $x\in \mathrm{m}_{S_{R}}(U)=m_{R}(U)$.
By condition (4) of Definition \ref{t esp mon}, $(z]\in R(x)$. Therefore,
$z\in S_{R}(x)$. 

Now, let $(x,Z)\in R$. We will see that $S_{R}(x)\cap Z\neq\emptyset$.
Since $\langle X,\mathcal{T},R\rangle$ is a normal monotonic $DS$-space,
there exists $z\in Z$ such that $(z]\in R(x)$. By definition $z\in S_{R}(x)\cap Z$.
Let $(x,Z)\in R_{S_{R}}$. Then, $S_{R}(x)\cap Z\neq\emptyset$. Let
$z\in Z$ such that $(x,z)\in S_{R}$. By definition of $S_{R}$,
$(x,(z])\in R$. So, $(z]\subseteq Z$ and by Proposition \ref{Rupset},
$(x,Z)\in R$. Therefore $R=R_{S_{R}}$.

2. Let $U\in\mathrm{Up}(X)$. We will prove that $m_{R_{S}}(U)=\mathrm{m}_{S}(U)$.
Let $x\in m_{R_{S}}(U)$ and $z\in S(x)$. Then, $S(x)\cap(z]\neq\emptyset$.
By definition of $R_{S}$, $(x,(z])\in R_{S}$. So, $(z]\cap U\neq\emptyset$
and since $U$ is an upset, we have that $z\in U$. Thus, $S(x)\subseteq U$
and $x\in \mathrm{m}_{S}(U)$. Let $x\in \mathrm{m}_{S}(U)$ and let $Z\in R_{S}(x)$.
Then, there exists $z\in Z$ such that $z\in S(x)$. So, $z\in U$
and we get that $Z\cap U\neq\emptyset$. Thus, $x\in m_{R_{S}}(U)$. 

So, we have that $\mathrm{m}_{S}(U)=m_{R_{S}}(U)\in D(X)$ for all $U\in D(X)$.
We will see that $\langle X,\mathcal{T},R_{S}\rangle$ is a normal
monotonic $DS$-space. Let $x\in X$, $Z\in\bigcap\{L_{U}:x\in m_{R_{S}}(U)\}$
and suppose that $Z\notin R_{S}(x)$. By definition, $Z\cap S(x)=\emptyset$
and since $S(x)$ is a closed subset, there exists $U\in D(X)$ such
that $Z\subseteq U^{c}$ and $S(x)\cap U^{c}=\emptyset$, i.e., $x\in \mathrm{m}_{S}(U)=m_{R_{S}}(U)$
and $Z\cap U=\emptyset$, which is a contradiction. Therefore, $R_{S}(x)=\bigcap\{L_{U}:x\in m_{R_{S}}(U)\}$.
Now, let $x\in X$ and $Z\in R_{S}(x)$. Then, there exists $z\in S(x)\cap Z$.
So, $S(x)\cap(z]\neq\emptyset$ and by definition of $R_{S}$, $(x,(z])\in R_{S}$. 

Finally, let $(x,z)\in S$. Then $z\in S(x)\cap(z]$ and, by definition,
$(x,(z])\in R_{S}$. Therefore, $(x,z)\in S_{R_{S}}$. On the other
hand, let $(x,z)\in S_{R_{S}}$. Then, $(x,(z])\in R_{S}$. So, $S(x)\cap(z]\neq\emptyset$.
Since $S$ is a meet-relation, $S(x)$ is an upset. Thus, $z\in S(x)$.
Therefore $S=S_{R_{S}}$.\end{proof} 

\begin{remark}Note that as a particular case we get the relation
defined in \cite{Gehrke} by Mai Gherke, where she gave an algebraic
derivation of the space associated to a bounded distributive lattice
with a modality $\square$ that preserves 1 and $\wedge$ based on
the canonical extension. Given the modality $\square:A\rightarrow A$,
since the extension $\square^{\sigma}=\square^{\pi}$ is completely
meet preserving, it is completely determined by its action on the
completely meet prime elements of the canonical extension. Working
on Stone spaces, the family $\mathcal{S}(X(\mathbf{A}))$ is the family
of all basic saturated sets and recall that $M^{\infty}(Up(X(\mathbf{A})))=\{\alpha(P^{c})^{c}=(P]^{c}:P\in X(\mathbf{A})\}$.
The relation $S\subseteq X(\mathbf{A})\times X(\mathbf{A})$ defined
in \cite{Gehrke} is: 
\[
(P,Q)\in S\Leftrightarrow\square^{\pi}((Q]^{c})\subseteq(P]^{c}.
\]
It is easy to see that $\square^{\pi}((Q]^{c})\subseteq(P]^{c}\Leftrightarrow\square^{-\text{1}}(P)\cap Q^c=\emptyset$.

\end{remark}

\subsection{Boolean Algebras with a monotonic operator}

\label{subsection: Boolean}

In paper \cite{Celaniboole} (see also \cite{Hansen} and \cite{HansenKupkePacuit})
S. Celani developed a topological duality between monotonic Boolean
algebras and descriptive monotonic frames. These frames are actually
monotonic $DS$-frames. 

Recall that a Boolean algebra with a normal monotonic operation, is
a pair $\langle\mathbf{A},\square\rangle$ such that $\mathbf{A}$
is a Boolean algebra, and $\square$ is an operator defined on $A$
such that 
\begin{enumerate}
\item $\square(a\wedge b)\leq\square a\wedge\square b$ for all $a,b\in A$, 
\item $\square1=1$.
\end{enumerate}
Also, recall that a \emph{Stone space} is a topological space $X=\langle X,\tau\rangle$ that is \textit{compact and totally disconnected}, i.e., given
distinct points $x,y\in X$, there exists a clopen (closed and open) subset
$U$ of $X$ such that $x\in U$ and $y\notin U$. Let $\mathrm{Clop}(X)$
be the family of closed and open subsets of a Stone space $\langle X,\tau\rangle$. 

A\emph{ descriptive} $m$-frame \cite{Celaniboole}, or \emph{monotonic
modal space}, is a triple $\langle X,R,\tau\rangle$ such that
\begin{enumerate}
\item $\langle X,\tau\rangle$ is a Stone space, 
\item $R\subseteq X\times\mathcal{C}_{0}(X)$, where $\mathcal{C}_{0}(X)=\mathcal{C}(X)-\{\emptyset\}$,
\item $\square_{R}(U)=\{x\in X:\forall Y\in R(x)\ (Y\cap U\neq\emptyset)\}\in\mathrm{Clop}(X)$
for all $U\in\mathrm{Clop}(X)$,
\item $R(x)=\bigcap\{L_{U}:x\in\square_{R}(U)\}$, for all $x\in X$. 
\end{enumerate}
\begin{remark} It is well known that if $\langle X,\tau\rangle$
is a Stone space then $\mathrm{Clop}(X)$ is a basis for the topology
and $\mathcal{S}(X)=\mathcal{C}(X)$. As $X$ is Hausdorff, the only
irreducible closed sets are singletons, so $X$ is sober. Then, it
is easy to see that any descriptive $m$-frame is a monotonic $DS$-space.\end{remark}

\begin{remark} Let $\mathbf{A}=\langle A,\vee,\wedge,\lnot,0,1\rangle$
be Boolean algebra. Note that if $F$ is a filter of $\mathbf{A}$,
then the set $I_{F}=\{\lnot a:a\in F\}$ is an ideal of $\mathbf{A}$,
and thus $\hat{F}=\alpha(I_{F})$. Similarly, if $I$ is an ideal
of $\mathbf{A}$, then the set $F_{I}=\{\lnot a:a\in I\}$ is a filter
of $\mathbf{A}$, and $\alpha(I)=\hat{F_{I}}$. \end{remark}

Let $\langle\mathbf{A},\square\rangle$ be a Boolean algebra endowed
with a normal monotonic operator $\square$. Let $\Diamond\colon A\rightarrow A$
be the dual operator defined by $\Diamond a=\lnot\square\lnot a$,
for each $a\in A$. Following the construction of the $\mathcal{S}$-
and $\mathcal{C}$- monotonic spaces, we have four relations: $G_{\Diamond}$,
$G_{\square}$, $R_{\Diamond}$ and $R_{\square}$. The following
proposition shows the relationships between them.

\begin{proposition} Let $\langle\mathbf{A},\square\rangle$ be a
Boolean algebra endowed with a monotonic operator. Then, $G_{\Diamond}=R_{\square}$
and $G_{\square}=R_{\Diamond}$.\end{proposition}

\begin{proof} We will prove that $G_{\Diamond}=R_{\square}$. Let
$F\in\mathrm{Fi}(\mathbf{A})$ such that $(P,\hat{F})\in G_{\Diamond}$.
Then, $F\subseteq\Diamond^{-1}(P)$, i.e., for all $a\in F,$ $\lnot\square(\lnot a)\in P$.
So, $\square(\lnot a)\notin P$, i.e., $\lnot a\notin\square^{-1}(P)$.
Thus, $\square^{-1}(P)\cap I_{F}=\emptyset$ and $(P,\alpha(I_{F}))\in R_{\square}$.
By remark, $\hat{F}=\alpha(I_{F})$. The proof of other inclusion is similar. 

The other equality is proved analogously.\end{proof}

\end{document}